\documentclass[12pt]{article}
\usepackage[left=1.2in,top=1in,right=1.2in,bottom=1in,letterpaper]{geometry}
\usepackage{hyperref}
\hypersetup{colorlinks=true, linkcolor=blue, citecolor=blue, urlcolor=blue, pdftitle={Block stochastic gradient method}, pdfauthor={Yangyang Xu}}

\usepackage[pdftex]{graphicx}

\usepackage[usenames]{color}

\usepackage{cleveref}

\usepackage{amsmath,amssymb,amsthm}
\usepackage{latexsym,amsfonts,amscd,amsxtra,amstext}
\usepackage{url}
\usepackage{cite}

\usepackage[ruled,vlined]{algorithm2e}
\usepackage{index}
\usepackage{subfigure}

\usepackage[normalem]{ulem} 

\newcommand{\va}{{\mathbf{a}}}

\newcommand{\vd}{{\mathbf{d}}}

\newcommand{\vg}{{\mathbf{g}}}
\newcommand{\vh}{{\mathbf{h}}}

\newcommand{\vu}{{\mathbf{u}}}

\newcommand{\vw}{{\mathbf{w}}}
\newcommand{\vx}{{\mathbf{x}}}
\newcommand{\vy}{{\mathbf{y}}}
\newcommand{\vz}{{\mathbf{z}}}

\newcommand{\vU}{{\mathbf{U}}}
\newcommand{\vV}{{\mathbf{V}}}

\newcommand{\vX}{{\mathbf{X}}}

\newcommand{\cA}{{\mathcal{A}}}

\newcommand{\cG}{{\mathcal{G}}}

\newcommand{\cI}{{\mathcal{I}}}

\newcommand{\cK}{{\mathcal{K}}}

\newcommand{\cM}{{\mathcal{M}}}
\newcommand{\cN}{{\mathcal{N}}}
\newcommand{\cO}{{\mathcal{O}}}
\newcommand{\cP}{{\mathcal{P}}}

\newcommand{\cX}{{\mathcal{X}}}

\newcommand{\EE}{\mathbb{E}} 
\newcommand{\RR}{\mathbb{R}} 
\newcommand{\vzero}{\mathbf{0}} 

\newcommand{\dist}{\mathrm{dist}}    
\newcommand{\Prob}{{\mathrm{Prob}}} 
\newcommand{\tr}{{\mathrm{tr}}} 
\DeclareMathOperator*{\argmin}{arg\,min} 

\newcommand{\bm}[1]{\boldsymbol{#1}}

\newcommand{\st}{\mbox{s.t.}}

\newcommand{\bc}{\begin{center}}
\newcommand{\ec}{\end{center}}

\newcommand{\bdm}{\begin{displaymath}}
\newcommand{\edm}{\end{displaymath}}

\newcommand{\beq}{\begin{equation}}
\newcommand{\eeq}{\end{equation}}

\newcommand{\bfl}{\begin{flushleft}}
\newcommand{\efl}{\end{flushleft}}

\newcommand{\bt}{\begin{tabbing}}
\newcommand{\et}{\end{tabbing}}

\newcommand{\beqn}{\begin{eqnarray}}
\newcommand{\eeqn}{\end{eqnarray}}

\newcommand{\beqs}{\begin{align*}} 
\newcommand{\eeqs}{\end{align*}}  

\newtheorem{theorem}{Theorem}

\newtheorem{assumption}{Assumption}

\newtheorem{corollary}{Corollary}
\newtheorem{remark}{Remark}
\newtheorem{lemma}{Lemma}

\newcount\refnum\refnum=0
\def\myref{{\global\advance\refnum by 1} {\bf \large Lecture \the \refnum. }}

\newcommand{\prox}{\mathbf{prox}}

\usepackage{color}

\begin{document}

\title{Block Stochastic Gradient Iteration\\ for Convex and Nonconvex Optimization
}

\author{Yangyang Xu\thanks{\url{yangyang.xu@rice.edu}, Computational and Applied Mathematics, Rice University, Houston, TX.}
\and Wotao Yin\thanks{\url{wotaoyin@math.ucla.edu}, Department of Mathematics, UCLA, Los Angeles, CA.}
}

\date{\today}

\maketitle

\begin{abstract}
The stochastic gradient (SG) method can quickly solve a problem with a large number of components in the objective, or a stochastic optimization problem, to a moderate accuracy. The block coordinate descent/update (BCD) method, on the other hand, can quickly solve  problems with multiple (blocks of) variables. This paper introduces a method that combines the  great features of SG and BCD for  problems with many components in the objective and with multiple (blocks of) variables.


This paper proposes  a block stochastic gradient (BSG) method for both convex and nonconvex programs.  BSG generalizes SG by   updating all the blocks of variables  in the Gauss-Seidel type (updating the current block depends on the previously updated block), in either a fixed or randomly shuffled order. Although BSG has slightly more work at each iteration, it typically  outperforms  SG because of BSG's Gauss-Seidel updates and larger stepsizes, the latter of which are determined by the smaller per-block Lipschitz constants.


The convergence of BSG  is established  for both convex and nonconvex cases. In the convex case, BSG has the same order of convergence rate as   SG. In the nonconvex case, its convergence is established in terms of the expected violation of a first-order optimality condition. In both cases our analysis is nontrivial since the typical unbiasedness assumption no longer holds. 

BSG is numerically evaluated on the following problems:  \emph{stochastic least squares} and \emph{logistic regression}, which are convex, and \emph{low-rank tensor recovery} and \emph{bilinear logistic regression}, which are nonconvex. On the convex problems, BSG  performed   significantly better than SG. On the nonconvex problems, 
 BSG   significantly outperformed the deterministic BCD method because the latter tends to early stagnate  near  local minimizers.  Overall, BSG inherits the benefits of both stochastic gradient approximation and block-coordinate updates and is especially useful for solving large-scale nonconvex problems.
\end{abstract}

\section{Introduction}
In many engineering and machine learning problems, we are facing optimization problems that involve a huge amount of data. 
It is often very expensive to use such a huge amount of data for every update of the problem variables, and a more efficient way is to sample a small amount from the collected data for each renewal of the variables.

Keeping this in mind, in this paper, we consider the stochastic program
\begin{equation}\label{eq:main}
\min_{\vx} \Phi(\vx)=\EE_\xi f(\vx;\xi)+\sum_{i=1}^s r_i(\vx_i), \ \st\ \vx_i\in\cX_i,\ i = 1,\ldots,s,
\end{equation}
where $\cX_i\subset\RR^{n_i},\forall i$, are convex constraint sets, the variable $\vx\in\RR^n$ is partitioned into disjoint blocks $\vx=(\vx_1,\ldots,\vx_s)$ of the dimension $n=\sum_{i=1}^s n_i$, $\xi$ is a random variable, $\EE_\xi f(\vx;\xi)$ is continuously differentiable, and $r_i$ are  regularization functions (possibly non-differentiable) such as $\ell_1$-norm $\|\vx_i\|_1$ or $\ell_0$ seminorm $\|\vx_i\|_0$ for sparse or other structured solutions. Throughout the paper, we let
\begin{align*}
F(\vx)=\EE_\xi f(\vx;\xi),\qquad
R(\vx)=\sum_{i=1}^sr_i(\vx_i),
\end{align*}
and for simplicity, we omit the subscript $\xi$ in the expectation operator without causing confusion.

Note that by assuming $\Prob\{\xi=\ell\}=\frac{1}{N},\ell = 1,\ldots,N$, \eqref{eq:main} includes as a special case the following deterministic program
\begin{equation}\label{eq:det}
\min_\vx \frac{1}{N}\sum_{\ell=1}^N f_\ell(\vx) +\sum_{i=1}^s r_i(\vx_i),\ \st \ \vx_i\in\cX_i,\forall i,
\end{equation}
where $N$ is often very large. Many problems in applications can be written in the form of \eqref{eq:main} or \eqref{eq:det} such as LASSO \cite{tibshirani1996regression}, sparse logistic regression \cite{shevade2003simple}, bilinear logistic regression \cite{dyrholm2007bilinear, shi2014sparse}, sparse dictionary learning \cite{mairal2009online}, low-rank matrix completion problem \cite{CandesRecht2008}, and so on.

We allow $F$ and $r_i$ to be nonconvex. When they are convex, we have sublinear convergence of the proposed method (see Algorithm \ref{alg:bsg}) in terms of objective value. Without convexity, we establish global convergence in terms of the expected violation of a first-order optimality condition. In addition, numerical experiments demonstrate that our algorithm can perform very well on both convex and nonconvex problems.

\subsection{Motivation}
One difficulty to solve \eqref{eq:main} is that it may be impossible or very expensive to accurately calculate the expectation to evaluate the objective and  its gradient or subgradient. One approach is the \emph{stochastic average approximation} (SAA) method \cite{kleywegt2002sample}, which generates a set of samples 
and then solves the empirical risk minimization problem
by a certain optimization method.

Another  approach  is the \emph{stochastic gradient} (SG) method (see \cite{robbins1951stochastic, polyak1990new, nemirovski2009robust} and the references therein), which assumes that a stochastic gradient $\vg(\vx;\xi)$ of $F$ can be obtained by a certain oracle and then iteratively performs the update
\begin{equation}\label{eq:sg}
\vx^{k+1}=\argmin_{\vx\in\cX}\, \langle \vg(\vx^k;\xi_k)+\tilde{\nabla} R(\vx^k),\vx-\vx^k\rangle +\frac{1}{2\alpha_k}\|\vx-\vx^k\|^2,
\end{equation}
where $\cX=\cX_1\times\ldots\times\cX_s$, and $\tilde{\nabla} R(\vx^k)$ is a subgradient of $R$ at $\vx^k$.
In \eqref{eq:sg}, $\xi_k$ is a realization of $\xi$ at the $k$th iteration, and $\alpha_k$ is a stepsize that is typically required to asymptotically reduce to \emph{zero} for convergence. The work \cite{nemirovski2009robust} compares SAA and SG and demonstrates that the latter is competitive and sometimes significantly outperforms the former  for solving a certain class of problems including the stochastic utility problem and stochastic max-flow problem. The SG method has also been popularly used (e.g., \cite{zhang2004solving, ShalevTewari2009, gemulla2011large, shalev2011pegasos, recht2013parallel}) to solve deterministic programming in the form of \eqref{eq:det} and exhibits advantages over the deterministic gradient method when $N$ is large and high solution accuracy is not required. 

To solve (deterministic) problems with separable nonsmooth terms as in \eqref{eq:det}, the block coordinate descent (BCD) method (see \cite{luo1992convergence, Grippo-Sciandrone-00, Tseng-01, TsengYun2009, wen2012block, xu2013block} and the references therein) has been widely used. At each iteration, BCD updates only one block of variables and thus can have a much lower per-iteration complexity than  methods updating all the variables together. BCD has been found efficient solving many large-scale problems (see \cite{ChangHsiehLin2008, nesterov2012efficiency, richtarik2012iteration, wen2012block, peng2013parallel} for example).

\subsection{Our algorithm} In order to take advantages of the structure of \eqref{eq:main} and maintain the benefits of BCD, we generalize SG to a \emph{block  stochastic gradient} (BSG) method,  which is given in Algorithm \ref{alg:bsg}.
\begin{algorithm}\caption{Block stochastic gradient for solving \eqref{eq:main}}\label{alg:bsg}
{\small
\DontPrintSemicolon
\textbf{Input:} starting point $\vx^1$, step sizes $\{\alpha_i^k:i=1,\ldots,s\}_{k=1}^\infty$, and positive integers $\{m_k\}_{k=1}^\infty$.\;
\For{$k=1,2,\ldots$}{
Generate a mini batch of samples $\Xi_k=\{\xi_{k,1},\xi_{k,2},\ldots,\xi_{k,m_k}\}$.\;
Define update order $\pi_i^k=i,\,i=1,\ldots,s$, or randomly shuffle $\{1,2,\ldots,s\}$ to $\{\pi_1^k,\pi_2^k,\ldots,\pi_s^k\}$.\;
\For{$i=1,\ldots,s$}{
Compute sample gradient for  the $\pi_i^k$th block
$$\tilde{\vg}_i^k=\frac{1}{m_k}\sum_{\ell=1}^{m_k}\nabla_{\vx_{\pi_i^k}}f(\vx_{\pi_{<i}^k}^{k+1},\vx_{\pi_{\ge i}^k}^k; \xi_{k,\ell}).$$
\If{$\cX_{\pi_i^k}=\RR^{n_{\pi_i^k}}$ (the $i$th block is unconstrained)}{Update the $\pi_i^k$th block
\begin{equation}\label{eq:update}
\vx_{\pi_i^k}^{k+1}=\argmin_{\vx_{\pi_i^k}}\langle \tilde{\vg}_i^k, \vx_{\pi_i^k}-\vx_{\pi_i^k}^k\rangle+\frac{1}{2\alpha_{{\pi_i}}^k}\|\vx_{\pi_i^k}-\vx_{\pi_i^k}^k\|^2+r_{\pi_i^k}(\vx_{\pi_i^k}).
\end{equation}
}
\Else{Update the $\pi_i^k$th block
\begin{equation}\label{eq:update2}
\vx_{\pi_i^k}^{k+1}=\argmin_{\vx_{\pi_i^k}\in\cX_{\pi_i^k}}\langle \tilde{\vg}_i^k+\tilde{\nabla} r_{\pi_i^k}(\vx_{\pi_i^k}^k), \vx_{\pi_i^k}-\vx_{\pi_i^k}^k\rangle+\frac{1}{2\alpha_{{\pi_i}}^k}\|\vx_{\pi_i^k}-\vx_{\pi_i^k}^k\|^2.
\end{equation}
}
}
}
}
\end{algorithm}

In the algorithm, we assume that samples of $\xi$ are randomly generated.  We let $\tilde{\vg}_i^k$
be a stochastic approximation of $\nabla_{\vx_{\pi_i^k}}F(\vx_{\pi_{<i}^k}^{k+1},\vx_{\pi_{\ge i}^k}^k)$,
where $\vx_{\pi_{<i}}$ is short for $(\vx_{\pi_1},\ldots,\vx_{\pi_{i-1}})$. In \eqref{eq:update2}, $\tilde{\nabla} r_j(\vx_j^k)$ is a subgradient of $r_j$ at $\vx_j^k$, and we assume it exists for all $j$ and $k$. We assume that both \eqref{eq:update} and \eqref{eq:update2} are easy to solve. We perform two different updates. 
When $\cX_i=\RR^{n_i}$, we prefer \eqref{eq:update} over \eqref{eq:update2} since proximal gradient iteration is typically faster than proximal subgradient iteration (see \cite{goffin1977convergence, boyd2003subgradient, nesterov2009primal, nesterov2014convergent} and the references therein); when $\cX_i\neq\RR^{n_i}$, we  use \eqref{eq:update2}, which takes the subgradient of $r_i$, since minimizing  the nonsmooth function $r_i$ subject to constraints is generally difficult.

One can certainly take $m_k=1$, $\forall k$, or use larger $m_k$'s. In general,  a larger $m_k$ leads to a  lower sample variance and incurs more computation of $\tilde{\vg}_i^k$. Note that at the beginning of each cycle, we allow  a reshuffle of the blocks, which can often lead to better overall numerical performance especially for nonconvex problems, as demonstrated in \cite{xu-yin-ebcd14}. For the convenience of our discussion and easy notation, we assume $\pi_i^k\equiv i,\,\forall i,k,$ throughout our analysis, i.e.,  all iterations of the algorithm update the blocks in the same ascending order. However, our analysis still goes through if the order is shuffled at the beginning of each cycle, and some of our numerical experiments use reshuffling. 

\subsection{Related work}
The BCD and SG methods are special cases of the proposed BSG method.
In Algorithm \ref{alg:bsg}, if $s=1$, i.e., there is only one block of variables, the update in \eqref{eq:update2} becomes the SG update \eqref{eq:sg}, and if $\pi_i^k=i$, and $\tilde{\vg}_i^k=\nabla_{\vx_i}F(\vx_{<i}^{k+1},\vx_{\ge i}^k), \forall i, k$, it becomes the BCD update in \cite{TsengYun2009, xu2013block}. For solving problem \eqref{eq:det}, a special case of \eqref{eq:main}, the deterministic BCD method in \cite{TsengYun2009, xu2013block} requires the partial gradients of all of the component functions for every block update while BSG uses only one or several of them, and thus the BSG update is much cheaper. On the other hand, BSG updates the variables in a Gauss-Seidel-like manner while SG does it in a Jacobi-like manner. Hence, BSG often takes fewer iterations than SG, and our numerical experiments demonstrate such an advantage of BSG over SG. For the reader's convenience, we list the related methods in Table \ref{table:abbr}.

\begin{table}\caption{List of related methods}\label{table:abbr}
\begin{center}
{\footnotesize
\begin{tabular}{|l|c|}\hline
\textbf{Abbreviation} & \textbf{Method}\\\hline
BSG (this paper) & block coordinate update using stochastic gradient\\\hline
SG \cite{nemirovski2009robust, lan2012optimal} & stochastic gradient\\\hline
BCD \cite{luo1992convergence, Tseng-01} & deterministic, block coordinate (minimization) descent\\\hline
BCGD \cite{TsengYun2009} & deterministic, block coordinate gradient descent\\\hline
SBCD \cite{nesterov2012efficiency} & stochastic block coordinate descent\\\hline
SBMD \cite{dang2013stochastic} & stochastic block mirror descent \\\hline
\end{tabular}}
\end{center}
\end{table}

The BCD method has a long history dating back to the 1950s \cite{Hildreth-57}, which considers strongly concave quadratic programming. Its original format is  block coordinate minimization (BCM), which cyclically updates all the blocks by minimizing the objective with respect to one block at a time whiling fixing all the others at their most recent values. The convergence of BCM has been  extensively analyzed in the literature for both convex and nonconvex cases; see \cite{luo1992convergence, GrippoSciandrone2000, Tseng-01, xu2013block, razaviyayn2013unified} for example. The work \cite{auslender1992asymptotic} combines the proximal point method \cite{rockafellar1976monotone, EcksteinBertsekas1992} with BCM and proposes a proximal block-coordinate update scheme. It was shown in \cite{navasca2008swamp} that such a scheme can perform better than the original BCM scheme for solving the tensor decomposition problem. Although BCD is straightforward to understand, there is no convergence rate known for general convex programming\footnote{The earlier work \cite{luo1992convergence} establishes the linear convergence of BCM by assuming strong convexity on the objective.} until \cite{nesterov2012efficiency}, which proposes a stochastic block-coordinate descent (SBCD) method. At each iteration, the SBCD method randomly chooses one block of variables and updates it by performing one proximal gradient update. SBCD is analyzed in \cite{nesterov2012efficiency} for the smooth convex case, and the analysis is generalized to non-smooth convex case in \cite{richtarik2012iteration, lu2013complexity}. It is shown that SBCD has the same order of convergence rate for the non-smooth case as it does in the smooth case. For general convex case, SBCD has sublinear convergence in terms of expected objective value, and for strongly convex case, it converges linearly.  SBCD has also been analyzed in \cite{lu2013randomized} for non-smooth nonconvex case. Recently, \cite{saha2013nonasymptotic, beck2013convergence} showed that cyclic BCD can have the same order of convergence rate as  SBCD for smooth convex programming, and \cite{hong2013iteration} then extended the work of \cite{beck2013convergence} to non-smooth convex case.

The SG method also has a long history and dates back to the pioneering work \cite{robbins1951stochastic}. Since then, the SG method becomes very popular in  stochastic programming. The classic analysis (e.g., in \cite{chung1954stochastic, sacks1958asymptotic}) requires second-order differentiability and also strong convexity, and under these assumptions, the method exhibits an asymptotically optimal convergence rate $\cO(1/k)$ in terms of  expected objective value. A great improvement was made in \cite{nemirovskici1983problem, polyak1990new}, which propose a robust SG method applicable to general convex problems and obtain a non-asymptotic convergence rate $\cO(1/\sqrt{k})$ by averaging all the iterates. These results were revisited in \cite{nemirovski2009robust}, which, in addition, proposes a mirror descent stochastic approximation method and achieves the same order of convergence rate with a better constant. Furthermore, the mirror descent method is accelerated in \cite{lan2012optimal, ghadimi2013optimal}, where the convergence results are strengthened for composite convex stochastic optimization. Recently, the SG method has been extended in \cite{ghadimi2013stochastic, ghadimi2013mini, ghadimi2013accelerated} to handle nonconvex problems, for which the convergence results are established in terms of the violation of first-order optimality conditions.

A very relevant work to ours is \cite{dang2013stochastic}, which proposes a \emph{stochastic block mirror descent} (SBMD) method by combining SBCD with the mirror descent stochastic approximation method. The main difference between SBMD and our proposed BSG is that at each iteration SBMD randomly chooses one block of variables to update while our BSG cyclically updates all the blocks of variables, the later updated blocks depending on the early updated blocks. We will demonstrate that BSG is competitive and often performs significantly better than SBMD when using the same number of samples. The practical advantage of BSG over SBMD should be intuitive and natural. For solving \eqref{eq:main}, both BSG and SBMD need samples or experimental observations of $\xi$. When the samples arrive sequentially, the computing processor can be idle between the arrivals of two samples if they only update  one block of variables  and thus wastes the computing resource. Technically, SBMD has unbiased stochastic block partial gradient by choosing one block randomly to update while cyclic block update of BSG results in biased partial gradient, and thus the analysis of BSG would be more challenging than that of SBMD. To get the same order of convergence rate for BSG, we will require stronger assumptions. Specifically, the analysis of SBMD in \cite{dang2013stochastic} does not assume boundedness of the iterates while our analysis of BSG requires such boundedness in expectation. SBMD also appeared in \cite{liu2014syn-random-Kaczmarz} for solving linear systems. Although the analysis in \cite{liu2014syn-random-Kaczmarz} assumes that one coordinate is updated at each iteration, its implementation updates all coordinates corresponding to the nonzeros of a randomly sampled row vector, and thus it explains somehow the benefit of updating all the blocks of variables instead of just one  at a time.

\subsection{Contributions}
We summarize our contributions as follows.
\begin{itemize}
\item We propose a BSG method for solving both convex and nonconvex optimization problems. The update order of the blocks of variables at each iteration is arbitrary and independent of other iterations; it can be fixed or shuffled. The new method inherits the benefits of both SG and BCD methods. It is applicable to stochastic programs in the form of \eqref{eq:main}, and it applies to deterministic problems in the form of \eqref{eq:det} with a huge amount of training data. It applies to problems with many variables. It allows both $N$ and $n$ to be large.
\item We analyze the BSG method for both convex and nonconvex problems. For convex problems, we show that it has the same order of convergence rate as that of the SG method, and for nonconvex problems, we establish its global convergence in terms of the expectation violation of first-order optimality conditions.
\item We demonstrate applying the BSG method to  two convex and  two nonconvex problems with both synthetic and real-world data, and we compared it to the stochastic methods SG and SBMD, and to the  deterministic method BCD. The numerical results demonstrate that it is at least comparable to, and  is often significantly better, than SG and SBMD on convex problems, and  it significantly outperforms the deterministic BCD on nonconvex problems. 
\end{itemize}

\subsection*{Notation}
Throughout the paper, we restrict our points in $\RR^n$ and use $\|\cdot\|$ for the Euclidean norm, but it is not difficult to generalize the analysis to any finite Hilbert space with a pair of primal and dual norms. We use $\vx_{<i}$ for $(\vx_1,\ldots,\vx_{i-1})$, $\nabla_{\vx_i}F(\vx)$ for the partial gradient of $F$ at $\vx$ with respect to $\vx_i$, and $\tilde{\nabla} r_i(\vx_i)$ as a subgradient in $\partial r(\vx_i)$, which is the limiting subdifferential (see \cite{rockafellar1998variational}). Scalars $\alpha_i^k, \alpha_k,\ldots$ are reserved for stepsizes and $L, L_{r_i},\ldots$ for Lipschitz constants. We let $\Xi_k$ denote the random set of samples generated at the $k$th iteration and $\bm{\Xi}_{[k]}=(\Xi_1,\ldots,\Xi_k)$ as the history of random sets from the 1st through  $k$th iteration.  $\EE[X|Y]$ denotes the expectation of $X$ conditional on $Y$. In addition, we partition the block set $\{1,2,\ldots,s\}$ to $\cI_1$ and $\cI_2$, where $$\cI_1=\{i: \vx_i\text{ is updated by }\eqref{eq:update}\}\qquad\text{and}\qquad \cI_2=\{i: \vx_i\text{ is updated by }\eqref{eq:update2}\}.$$

Given any set $\cX$, we let
$$\iota_\cX(\vx)=\left\{\begin{array}{ll}
0, &\text{ if }\vx\in\cX,\\
+\infty, &\text{ otherwise,}
\end{array}\right.$$
be the indicator function of $\cX$. For any function $r$ and any convex set $\cX$, we let
$$\prox_{\alpha r}(\vy)=\argmin_\vx r(\vx)+\frac{1}{2\alpha}\|\vx-\vy\|^2$$
be the proximal mapping, and
$$\cP_\cX(\vy)=\argmin_{\vx\in\cX} \|\vx-\vy\|^2$$
be the projection onto $\cX$. When $r$ is convex, $\prox_{\alpha r}$ is nonexpansive for any $\alpha>0$, i.e.,
\begin{equation}\label{eq:proxnx}\|\prox_{\alpha r}(\vx)-\prox_{\alpha r}(\vy)\|\le \|\vx-\vy\|,\ \forall \vx,\vy.
\end{equation} Note that $\cP_\cX=\prox_{\iota_\cX}$, and thus $\cP_\cX$ is nonexpansive for any convex set $\cX$.
Other notation will be specified when they appear.

\section{Convergence analysis}
In this section, we analyze the convergence of Algorithm \ref{alg:bsg} under different settings. Without loss of generality, we assume a fixed update order in Algorithm \ref{alg:bsg}:
$$\pi_i^k=i,\quad \forall i,k,$$
since the analysis still holds otherwise.
Hence, we have
$$\tilde{\vg}_i^k=\frac{1}{m_k}\sum_{\ell=1}^{m_k}\nabla_{\vx_i}f(\vx_{<i}^{k+1},\vx_{\ge i}^k;\xi_{k,\ell}).$$
Define
\begin{align*}
&\vg_i^k=\nabla_{\vx_i} F(\vx_{<i}^{k+1},\vx_{\ge i}^k),\\
&\bm{\delta}_i^k=\tilde{\vg}_i^k-\vg_i^k.
\end{align*}
Note that although $\vg_i^k$ is defined using the full function $F$ and does not explicitly depend on the random sample set $\Xi_k$, it in fact does depend on $\Xi_k$ since $\vx_j^{k+1},\forall j<i$, depend on $\Xi_k$. This is a big difference between BSG and SG and makes our analysis much more challenging.
In our analysis below, some of the following assumptions will be made for each result.
\begin{assumption}\label{assump1}
There exist a constant $A$ and a sequence $\{\sigma_k\}$ such that for any $i$ and $k$,
\begin{subequations}\label{eq:assump1}
\begin{align}
&\big\|\EE \big[\bm{\delta}_i^k|\bm{\Xi}_{[k-1]}\big]\big\|\le A\cdot\max_j \alpha_j^k,\label{eq:mean}\\
&\EE \|\bm{\delta}_i^k\|^2\le \sigma^2_k.\label{eq:var}
\end{align}
\end{subequations}
\end{assumption}

\begin{assumption}\label{assump2}
The objective function is lower bounded, i.e., $\Phi(\vx)>-\infty$.
There is a uniform Lipschitz constant $L>0$ such that
\begin{equation}\label{eq:lipobj}
\|\nabla_{\vx_i}F(\vx)-\nabla_{\vx_i}F(\vy)\|\le L\|\vx-\vy\|,\quad \forall \vx,\vy,~\forall i.
\end{equation}
\end{assumption}

\begin{assumption}\label{assump3}
There exists a constant $\rho$ such that $\EE\|\vx^k\|^2\le \rho^2$ for all $k$.
\end{assumption}

\begin{assumption}\label{assump4}
Every  function $r_i$ is Lipschitz continuous, namely, there is a constant $L_{r_i}$ such that
\begin{equation*}\label{eq:lipr}
\|r_i(\vx_i)-r_i(\vy_i)\|\le L_{r_i}\|\vx_i-\vy_i\|,~ \forall \vx_i, \vy_i.
\end{equation*}
We let
$$L_{\max}=\max_i L_{r_i}$$
be the dominant Lipschitz constant.
\end{assumption}
\begin{remark}
In Assumption \ref{assump1}, we assume  a bounded $\EE[\bm{\delta}_i^k|\bm{\Xi}_{k-1}]$ in \eqref{eq:mean} since the common assumption $\EE[\bm{\delta}_i^k|\bm{\Xi}_{k-1}]=\vzero$  fails to hold in our algorithm. Since the block updates in Algorithm \ref{alg:bsg} are Gauss-Seidel, the gradient error $\bm{\delta}_i^k$ typically has a nonlinear dependence on $\vx_j^{k+1},\forall  j<i$, which depends on $\bm{\Xi}_{k}$. On the other hand,  the boundedness assumption \eqref{eq:mean} holds {under proper
conditions}. For example, in \eqref{eq:main}, let $f(\vx;\xi)=f_\xi(\vx)$ with $\Prob\{\xi=\ell\}=\frac{1}{N}, \ell=1,\ldots,N$. Then $F(\vx)=\frac{1}{N}\sum_{\ell=1}^N f_\ell(\vx)$. Assume that $r_i$'s are convex and $\cX_i=\RR^{n_i},\forall i$, and that $f_\ell$ has Lipschitz continuous partial gradient with a uniform constant $L$, namely,
$$\|\nabla_{\vx_i} f_\ell(\vx)-\nabla_{\vx_i} f_\ell(\vy)\|\le L\|\vx-\vy\|,\quad \forall i, \ell,\forall \vx,\vy.$$
In addition, assume that each $\Xi_k$ is a singleton, i.e., $\Xi_k=\{\xi_k\}$ with $\xi_k$ uniformly selected from $\{1,\ldots,N\}$. Then
$$\EE[\tilde{\vg}_i^k|\bm{\Xi}_{[k-1]}]=\sum_{\ell=1}^N\Prob\{\xi_k=\ell\}\nabla_{\vx_i}f_\ell(
\vy_{<i}^{\ell,k+1},\vx_{\ge i}^k)=\frac{1}{N}\sum_{\ell=1}^N\nabla_{\vx_i}f_\ell(
\vy_{<i}^{\ell,k+1},\vx_{\ge i}^k),
$$
where $\vy_1^{\ell,k+1}=\prox_{\alpha_1^k r_1}\big(\vx_1^k-\alpha_1^k\nabla_{\vx_1}f_\ell(\vx^k)\big)
$ and $\vy_p^{\ell,k+1}=\prox_{\alpha_p^k r_p}\big(\vx_p^k-\alpha_p^k\nabla_{\vx_p}f_\ell(\vy_{<p}^{\ell,k+1}
,\vx_{\ge p}^k)\big)$ for $p\ge 2$, and
\begin{align*}
\EE[\vg_i^k|\bm{\Xi}_{[k-1]}]=&\sum_{m=1}^N\Prob\{\xi_k=m\}\nabla_{\vx_i}F(\vy_{<i}^{m,k+1},\vx_{\ge i}^k)&\qquad(\text{note }\vg_i^k\text{ depends on }\xi_k)\\
=&\frac{1}{N^2}\sum_{m=1}^N\sum_{\ell=1}^N\nabla_{\vx_i}f_\ell(\vy_{<i}^{m,k+1},\vx_{\ge i}^k)&\qquad(\text{Substitute }F)\\
=&\frac{1}{N^2}\sum_{\ell=1}^N\sum_{m=1}^N\nabla_{\vx_i}f_\ell(\vy_{<i}^{m,k+1},\vx_{\ge i}^k).
\end{align*}
Combining the above formulas of $\EE[\tilde{\vg}_i^k|\bm{\Xi}_{[k-1]}]$ and $\EE[\vg_i^k|\bm{\Xi}_{[k-1]}]$ gives
\begin{align*}
\big\|\EE[\bm{\delta}_i^k|\bm{\Xi}_{[k-1]}]\big\|=&\big\|\EE[\tilde{\vg}_i^k-\vg_i^k |\bm{\Xi}_{[k-1]}]\big\|\\
\le&\frac{1}{N^2}\sum_{\ell=1}^N\sum_{m=1}^N\big\|\nabla_{\vx_i}f_\ell(
\vy_{<i}^{\ell,k+1},\vx_{\ge i}^k)-\nabla_{\vx_i}f_\ell(\vy_{<i}^{m,k+1},\vx_{\ge i}^k)\big\|\\
\le &\frac{L}{N^2}\sum_{\ell=1}^N\sum_{m=1}^N\|\vy_{<i}^{\ell,k+1}-\vy_{<i}^{m,k+1}\|\qquad (\text{ from gradient Lipschitz continuity of }f_\ell)\\
\le &\frac{L}{N^2}\sum_{\ell=1}^N\sum_{m=1}^N\sum_{j<i}\|\vy_{j}^{\ell,k+1}-\vy_{j}^{m,k+1}\|\\
\le&\frac{L}{N^2}\sum_{\ell=1}^N\sum_{m=1}^N\sum_{j<i}\alpha_j^k\big\|\nabla_{\vx_j}f_\ell(\vy_{<j}^{\ell,k+1}
,\vx_{\ge j}^k)-\nabla_{\vx_j}f_m(\vy_{<j}^{m,k+1}
,\vx_{\ge j}^k)\big\|,
\end{align*}
where the last inequality is from the nonexpansiveness of the proximal mapping in \eqref{eq:proxnx}.
Therefore, if $\|\nabla_{\vx_i}f_\ell(\vx)\|\le M, \forall i, \ell, \forall \vx$, then we have from the above inequality that
$$\big\|\EE[\bm{\delta}_i^k|\bm{\Xi}_{k-1}]\big\|\le 2LsM\max_j\alpha_j^k.$$

The second condition \eqref{eq:var} in Assumption \ref{assump1} is standard in the literature of stochastic gradient method. The variance bound $\sigma_k^2$ is allowed to vary with the iteration $k$. 
\end{remark}

\begin{remark}\label{rm:bditer}
Assumption \ref{assump3} is relatively  weaker than the assumption made in the literature of stochastic gradient method; for example, \cite{nemirovski2009robust, lan2012optimal}  assume $\vx\in\cX$ for a bounded set $\cX$. Assumption \ref{assump3} is needed because the stochastic partial gradient $\tilde{\vg}_i^k$ may be biased (see Remark \ref{rm:cvx-rst}). Our analysis below for nonconvex case will remain valid if the assumption is weakened from the boundedness of $\{\EE\|\vx^k\|^2\}$ to that of $\{\EE\|\nabla F(\vx^k)\|^2\}$. 

Note that Assumption \ref{assump3} together with the partial gradient Lipschitz continuity of $F$ in Assumption \ref{assump2} implies the boundedness of $\{\EE\|\nabla F(\vx^k)\|^2\}$ by the following argument
$$\|\nabla_{\vx_i}F(\vx^k)\|^2\le 2\|\nabla_{\vx_i}F(\vx^k)-\nabla_{\vx_i}F(\vzero)\|^2+2\|\nabla_{\vx_i}F(\vzero)\|^2\le 2L^2\|\vx^k\|^2+2\|\nabla_{\vx_i}F(\vzero)\|^2.$$
Throughout the paper, we let
\begin{equation}\label{eq:defM}
M_\rho=\sqrt{4L^2\rho^2+2\max_i\|\nabla_{\vx_i}F(\vzero)\|^2},
\end{equation}
and thus we have $\EE\|\nabla_{\vx_i}F(\vx_{<i}^{k+1},\vx_{\ge i}^k)\|^2\le M_\rho^2,\forall i, k$. In addition, by Jensen's inequality, we have $\EE\|\vx^k\|\le \rho$, and $\EE\|\nabla_{\vx_i}F(\vx_{<i}^{k+1},\vx_{\ge i}^k)\|\le M_\rho, \forall i,k$.
\end{remark}

\begin{remark}\label{rm:exlip}
Assumption \ref{assump4} is the same as that in \cite{lan2012optimal}. One example is $r_i(\vx_i)=\lambda_i\|\vx_i\|_1$, which satisfies \eqref{eq:lipr} with $L_{r_i}=\lambda_i\sqrt{n_i}$. Note that \eqref{eq:lipr} implies that $\|\vd_i\|\le L_{r_i}$ for any $\vd_i\in\partial r_i(\vx_i)$.
\end{remark}

We first establish some lemmas, whose proofs are given in Appendix \ref{sec:app}. They are not difficult to prove and are useful in our convergence analysis.
\begin{lemma}\label{lem:indp}
Let $\vu^k$  be a random vector depending on $\bm{\Xi}_{[k-1]}$. Under Assumption \ref{assump1}, if $\vu^k$ is independent of $\bm{\delta}_i^k$ conditional on $\bm{\Xi}_{[k-1]}$, then
\begin{equation}\label{eq:indp}
\EE\langle \vu^k, \bm{\delta}_i^k\rangle\le A\big(\max_j \alpha_j^k\big)\EE\|\vu^k\|.
\end{equation}
\end{lemma}

\begin{lemma}\label{lem:bdh}Let
\begin{equation}\label{eq:gradmap}
\tilde{\vh}_i^k=\frac{1}{\alpha_i^k}(\vx_i^k-\vx_i^{k+1})
\end{equation}
be the stochastic gradient mapping for the $i$th block at the $k$th iteration.
Under Assumption \ref{assump4}, we have
$$\EE\|\tilde{\vh}_i^k\|^2\le 2\EE\|\tilde{\vg}_i^k\|^2+2L_{r_i}^2\le 4M_\rho^2+4\sigma_k^2+2L_{r_i}^2.$$

\end{lemma}

\begin{lemma}\label{lem:subrate}
If a nonnegative scalar sequence $\{A_k\}_{k\ge 1}$ satisfies
$$A_{k+1}\le \left(1-\frac{a}{k}\right)A_k+\frac{b}{k^2},$$
where $a,b>0$,
then it obeys
\begin{equation}\label{eq:subrate}
A_k\le \frac{c}{k},\quad \forall k\ge \lfloor a \rfloor+1,
\end{equation}
with
$$c=\left\{\begin{array}{ll}
2b/(a-1), &\text{ if }a>1,\\
\max\{b/a, A_1\}, & \text{ if } a\le 1,
\end{array}\right.$$
where $\lfloor a\rfloor$ denotes the largest integer that is no greater than $a$.
\end{lemma}

Next we present the convergence results of Algorithm \ref{alg:bsg} for both convex and nonconvex cases. For the convex case, where functions $F$ and $r_i$'s are convex, we establish a sublinear convergence rate in the same order of the SG method. For the nonconvex case, we establish the convergence result in terms of the expected violation of first-order optimality conditions.
\subsection{Convex case}
In this section, we first analyze Algorithm \ref{alg:bsg} for general convex problems and then for strongly convex problems to obtain stronger results.

\begin{theorem}[Ergodic convergence for non-smooth convex case]\label{thm:ns-cvx}
Let $\{\vx^k\}$ be generated from Algorithm \ref{alg:bsg} with $\alpha_i^k=\alpha_k=\frac{\theta}{\sqrt{k}}<\frac{1}{L},\forall i, k$, for some positive constant $\theta<\frac{1}{L}$. Under Assumptions \ref{assump1} through \ref{assump4}, if $F$ and $r_i$'s are all convex, $\vx^*$ is a solution of \eqref{eq:main} and $\sigma=\sup_k\sigma_k<\infty$, then
\begin{equation}\label{eq:ns-cvx-rate}
\EE [\Phi(\tilde{\vx}^K)-\Phi(\vx^*)]\le D\theta\frac{1+\log K}{\sqrt{1+K}}+\frac{\|\vx^*-\vx^1\|^2}{2\theta\sqrt{1+K}},
\end{equation}
where $\tilde{\vx}^K=\frac{\sum_{k=1}^K\alpha_k\vx^{k+1}}{\sum_{k=1}^K\alpha_k}$ and
\begin{equation}\label{eq:defD}D=\frac{s(\sigma^2+4L_{\max}^2)}{1-L\theta}+\sqrt{s}(\|\vx^*\|+\rho)\left(A+L\sqrt{\sum_{j=1}^s(4M_\rho^2+4\sigma_k^2+2L_{r_j}^2)}\right).
\end{equation}
Furthermore, if the maximum number of iterations $K$ is predetermined, then taking $\alpha_k=\frac{\theta}{\sqrt{K}},\forall k$, we have
\begin{equation}\label{eq:betterrate}\EE[\Phi(\tilde{\vx}^K)-\Phi(\vx^*)]\le\frac{D\theta}{\sqrt{K}}+\frac{\|\vx^*-\vx^1\|^2}{2\theta \sqrt{K}}.
\end{equation}
\end{theorem}

\begin{remark}
The value of $\theta$ usually plays a vital role on the actual speed  of the algorithm. In practice, we may not know the exact value of $L$, and thus it is difficult to choose an appropriate $\theta$. Even if we know or can estimate $L$, different $\theta$'s can make the algorithm perform very differently. This phenomenon has also been observed for the SG method; see numerical tests in section \ref{sec:numerical} and also the discussion on page 6 of \cite{nemirovski2009robust}. The work \cite{schaul2013no} and its references study adaptive learning rates, which is beyond the discussion of this paper.
\end{remark}

\begin{remark}\label{rm:cvx-rst}
The proof below will clarify that  the last (long) term in \eqref{eq:defD} is a result of the biased stochastic partial gradient (see \eqref{eq:ns-cvx-6-1} and \eqref{eq:ns-cvx-6} below). If unbiased, that is, $\EE[\bm{\delta}_i^k\,|\,\bm{\Xi}_{(k-1)}]=\vzero$ holds instead of \eqref{eq:mean}, we will have the improved $D=\frac{s(\sigma^2+4L_{\max}^2)}{1-L\theta}$ instead of \eqref{eq:defD}, and then the result in \eqref{eq:betterrate} becomes comparable 
to that of SG  since $\sigma^2$ is a bound of the \emph{block partial} stochastic gradient and thus $s\sigma^2$ a bound of the stochastic gradient of $F$.

Although the existence of the second term in \eqref{eq:defD} makes the result in \eqref{eq:betterrate} seemingly worse than that of SG, multi-block $(s>1)$ Gauss-Seidel-type updates are generally more effective. Note that BSG has a per-iteration cost similar to that of SG. To update the first block, computing the sample partial gradient requires reading the current values of all the blocks. The subsequent updates in BSG are much cheaper because the sample partial gradients can be updated from the ones already computed. In addition, BSG can take greater stepsizes than SG (because $L$ in \eqref{eq:lipobj} is the Lipschitz constant of partial gradients). Therefore, BSG can perform much better than SG as shown by numerical results in section \ref{sec:numerical}.

Minimizing the right-hand side of \eqref{eq:betterrate} by setting $\theta=\frac{\|\vx^*-\vx^1\|}{\sqrt{2D}}$, we have
\begin{equation}\label{eq:optrate}
\EE[\Phi(\tilde{\vx}^K)-\Phi(\vx^*)]\le\frac{\sqrt{2D}\|\vx^*-\vx^1\|}{\sqrt{K}}.
\end{equation}
 Note that if $\EE[\bm{\delta}_i^k\,|\,\bm{\Xi}_{(k-1)}]=\vzero$, the complexity in \eqref{eq:optrate} becomes just a fraction $\cO(\frac{1}{\sqrt{s}})$ of \cite[(3.22)]{dang2013stochastic} since the quantity $s\sigma^2$ in \eqref{eq:defD} equals $\sigma^2$ in \cite[(3.22)]{dang2013stochastic}, and $L_{\max}$ disappears from $D$ under the settings of \cite{dang2013stochastic}. Since the SBMD method of \cite{dang2013stochastic} only performs one block coordinate update at each iteration, its overall complexity  is as good as ours. But, once again BSG updates all the blocks while SBMD updates just a random one. Computationally, the cost of computing the sample partial gradients for all the blocks in BSG is dominated by the first   one, which involves the current values of all the blocks and which is also needed  by SBMD.  Therefore, at each iteration, BSG and SBMD spend the same cost to update a block, but  BSG then updates  the rest of the blocks at very little extra cost. Therefore, BSG can have better overall performance (see numerical results in section \ref{sec:logreg}).
\end{remark}

\begin{proof} From the Lipschitz continuity of $\nabla_{\vx_i}F(\vx_{<i}^{k+1},\vx_i,\vx_{>i}^k)$ about $\vx_i$, it holds for any $i\in\cI_1$ that (see \cite{NesterovConvexBook2004} for example)
\begin{align}\label{eq:ns-cvx-1}
&\Phi(\vx_{\le i}^{k+1},\vx_{>i}^k)-\Phi(\vx_{<i}^{k+1},\vx_{\ge i}^k)\cr
\le & \langle \vg_i^k,\vx_i^{k+1}-\vx_i^k\rangle +\frac{L}{2}\|\vx_i^{k+1}-\vx_i^k\|^2+r_i(\vx_i^{k+1})-r_i(\vx_i^k).
\end{align}
In addition, from Lemma 2 of \cite{lan2012optimal}, it holds for $i\in\cI_2$ that
\begin{align}\label{eq:ns-cvx-1-2}
&\Phi(\vx_{\le i}^{k+1},\vx_{>i}^k)-\Phi(\vx_{<i}^{k+1},\vx_{\ge i}^k)\cr
\le & \langle \vg_i^k+\tilde{\nabla} r_i(\vx_i^k),\vx_i^{k+1}-\vx_i^k\rangle +\frac{L}{2}\|\vx_i^{k+1}-\vx_i^k\|^2+2L_{r_i}\|\vx_i^{k+1}-\vx_i^k\|.
\end{align}
By Lemma 2 of \cite{BeckTeboulle2009}, the updates in \eqref{eq:update} and \eqref{eq:update2} indicate that for any $\vx_i\in\cX_i$, if $i\in\cI_1$, then
\begin{align}\label{eq:ns-cvx-2}
&\langle \tilde{\vg}_i^k, \vx_i^{k+1}-\vx_i^k\rangle+r_i(\vx_i^{k+1})+\frac{1}{2\alpha_k}\|\vx_i^{k+1}-\vx_i^k\|^2\cr
\le & \langle\tilde{\vg}_i^k, \vx_i-\vx_i^k\rangle+r_i(\vx_i)+\frac{1}{2\alpha_k}\big(\|\vx_i-\vx_i^k\|^2-\|\vx_i-\vx_i^{k+1}\|^2\big),
\end{align}
and if $i\in\cI_2$, then
\begin{align}\label{eq:ns-cvx-2-2}
&\langle \tilde{\vg}_i^k+\tilde{\nabla} r_i(\vx_i^k), \vx_i^{k+1}-\vx_i^k\rangle+\frac{1}{2\alpha_k}\|\vx_i^{k+1}-\vx_i^k\|^2\cr
\le & \langle\tilde{\vg}_i^k+\tilde{\nabla} r_i(\vx_i^k), \vx_i-\vx_i^k\rangle+\frac{1}{2\alpha_k}\big(\|\vx_i-\vx_i^k\|^2-\|\vx_i-\vx_i^{k+1}\|^2\big),
\end{align}
Summing up \eqref{eq:ns-cvx-1} through \eqref{eq:ns-cvx-2-2} over $i$ and arranging terms, we have
\begin{align}\label{eq:ns-cvx-3}
\Phi(\vx^{k+1})-\Phi(\vx^k)
\le &\sum_{i=1}^s\langle \vg_i^k-\tilde{\vg}_i^k, \vx_i^{k+1}-\vx_i^k\rangle-\big(\frac{1}{2\alpha_k}-\frac{L}{2}\big)\|\vx^{k+1}-\vx^k\|^2\\
&+\frac{1}{2\alpha_k}\big(\|\vx-\vx^k\|^2-\|\vx-\vx^{k+1}\|^2\big)
+\sum_{i\in\cI_1}\left(\langle\tilde{\vg}_i^k, \vx_i-\vx_i^k\rangle+r_i(\vx_i)-r_i(\vx^k_i)\right)\cr
& + \sum_{i\in\cI_2}\left(\langle\tilde{\vg}_i^k+\tilde{\nabla} r_i(\vx_i^k), \vx_i-\vx_i^k\rangle+2L_{r_i}\|\vx_i^{k+1}-\vx_i^k\|\right)\nonumber.
\end{align}
From the convexity of $F$ and $r_i$'s, we have $\Phi-\sum_{i\in \cI_1}r_i$ to be convex and thus
$$\Phi(\vx^k)-\sum_{i\in\cI_1}r_i(\vx^k_i)\le \Phi(\vx)-\sum_{i\in\cI_1}r_i(\vx_i)+\langle \nabla F(\vx^k),\vx^k-\vx\rangle+\sum_{i\in\cI_2}\langle\tilde{\nabla} r_i(\vx_i^k),\vx_i^k-\vx_i\rangle,$$
which together with \eqref{eq:ns-cvx-3} gives
\begin{align}\label{eq:ns-cvx-4}
\Phi(\vx^{k+1})-\Phi(\vx)
\le &\sum_{i=1}^s\langle \vg_i^k-\tilde{\vg}_i^k, \vx_i^{k+1}-\vx_i^k\rangle-\big(\frac{1}{2\alpha_k}-\frac{L}{2}\big)\|\vx^{k+1}-\vx^k\|^2+\sum_{i\in\cI_2}2L_{r_i}\|\vx_i^{k+1}-\vx_i^k\|\cr
& +\sum_{i=1}^s\langle\tilde{\vg}_i^k-\nabla_{\vx_i}F(\vx^k), \vx_i-\vx_i^k\rangle +\frac{1}{2\alpha_k}\big(\|\vx-\vx^k\|^2-\|\vx-\vx^{k+1}\|^2\big)\nonumber\\
(\text{note }\tilde{\vg}_i^k=\vg_i^k+\bm{\delta}_i^k)\quad= &\sum_{i=1}^s\langle -\bm{\delta}_i^k, \vx_i^{k+1}-\vx_i^k\rangle-\big(\frac{1}{2\alpha_k}-\frac{L}{2}\big)\|\vx^{k+1}-\vx^k\|^2+\sum_{i\in\cI_2}2L_{r_i}\|\vx_i^{k+1}-\vx_i^k\|\\
& +\sum_{i=1}^s\langle\bm{\delta}_i^k+{\vg}_i^k-\nabla_{\vx_i}F(\vx^k), \vx_i-\vx_i^k\rangle +\frac{1}{2\alpha_k}\big(\|\vx-\vx^k\|^2-\|\vx-\vx^{k+1}\|^2\big)\nonumber.
\end{align}
Noting $\frac{1}{\alpha_k}-L>0$ and using Young's inequality $ab< ta^2+\frac{1}{4t}b^2$ in the following two inequalities, we have
\begin{align}\label{eq:ns-cvx-5}
&\sum_{i=1}^s\langle -\bm{\delta}_i^k, \vx_i^{k+1}-\vx_i^k\rangle-\big(\frac{1}{2\alpha_k}-\frac{L}{2}\big)\|\vx^{k+1}-\vx^k\|^2+\sum_{i\in\cI_2}2L_{r_i}\|\vx_i^{k+1}-\vx_i^k\|\cr
\le & \sum_{i=1}^s\left(\langle -\bm{\delta}_i^k, \vx_i^{k+1}-\vx_i^k\rangle-\frac{1}{4}\big(\frac{1}{\alpha_k}-L\big)\|\vx^{k+1}_i-\vx^k_i\|^2\right)\cr
&+\sum_{i\in\cI_2}\left(2L_{r_i}\|\vx_i^{k+1}-\vx_i^k\|-\frac{1}{4}\big(\frac{1}{\alpha_k}-L\big)\|\vx^{k+1}_i-\vx^k_i\|^2\right)\cr
\le & \sum_{i=1}^s\frac{\alpha_k}{1-L\alpha_k}\|\bm{\delta}_i^k\|^2+\sum_{i\in\cI_2}\frac{4L_{r_i}^2\alpha_k}{1-L\alpha_k}.
\end{align}
In addition, letting $\vu^k=\vx_i-\vx_i^k$ in Lemma \ref{lem:indp}, we have
\begin{equation}\label{eq:ns-cvx-6-1}
\EE\langle \bm{\delta}_i^k, \vx_i-\vx_i^k\rangle
\le A\alpha_k\EE\|\vx_i-\vx_i^k\|,
\end{equation}
and also by H{\"{o}}lder's inequality, it holds that
\begin{align}\label{eq:ns-cvx-6}
\EE \langle{\vg}_i^k-\nabla_{\vx_i}F(\vx^k), \vx_i-\vx_i^k\rangle\le&\sqrt{\EE\|\vg_i^k-\nabla_{\vx_i}F(\vx^k)\|^2}\cdot\sqrt{\EE\|\vx_i-\vx_i^k\|^2}\cr
\le & L\alpha_k\sqrt{\EE\sum_{j=1}^s\|\tilde{\vh}_j^k\|^2}\cdot\sqrt{\EE\|\vx_i-\vx_i^k\|^2}\cr
\le &  L\alpha_k\sqrt{\sum_{j=1}^s(4M_\rho^2+4\sigma_k^2+2L_{r_j}^2)}\cdot\sqrt{\EE\|\vx_i-\vx_i^k\|^2},
\end{align}
where we have used Lemma \ref{lem:bdh} in the last inequality.

Taking expectation over both sides of \eqref{eq:ns-cvx-4},
substituting \eqref{eq:var}, \eqref{eq:ns-cvx-5}, \eqref{eq:ns-cvx-6-1}, and \eqref{eq:ns-cvx-6} into it, and noting $$\sum_{i=1}^s\sqrt{\EE\|\vx_i-\vx_i^k\|^2}\le \sqrt{s}\sqrt{\EE\|\vx-\vx^k\|^2}\le\sqrt{s}(\|\vx\|+\rho),$$ we have after some arrangement that
\begin{align}\label{eq:ns-cvx-7}
& \alpha_k\EE[\Phi(\vx^{k+1})-\Phi(\vx)]\cr
\le & \left(\frac{s(\sigma^2+4L_{\max}^2)}{1-L\alpha_k}+\sqrt{s}(\|\vx\|+\rho)\left(A+L\sqrt{\sum_{j=1}^s(4M_\rho^2+4\sigma_k^2+2L_{r_j}^2)}\right)\right)\alpha_k^2\\
&+\frac{1}{2}\EE\big(\|\vx-\vx^k\|^2-\|\vx-\vx^{k+1}\|^2\big)\nonumber.
\end{align}
Letting $\vx=\vx^*$ in the above inequality and taking  its sum over $k$, we have
\begin{align}\label{eq:ns-cvx-8}
\sum_{k=1}^K\alpha_k\EE[\Phi(\vx^{k+1})-\Phi(\vx^*)]
\le D \sum_{k=1}^K\alpha_k^2+ \frac{1}{2}\|\vx^*-\vx^1\|^2.
\end{align}
Now use the convexity of $\Phi$ and $\alpha_k=\frac{\theta}{\sqrt{k}}$ in \eqref{eq:ns-cvx-8} to get
\begin{align}\label{eq:ns-cvx-9}
\EE [\Phi(\tilde{\vx}^K)-\Phi(\vx^*)]\le & D \frac{\sum_{k=1}^K\alpha_k^2}{\sum_{k=1}^K\alpha_k}+ \frac{\|\vx^*-\vx^1\|^2}{2\sum_{k=1}^K\alpha_k}\\
\le & D\theta\frac{1+\log K}{\sqrt{1+K}}+\frac{\|\vx^*-\vx^1\|^2}{2\theta\sqrt{1+K}}.\nonumber
\end{align}
If the maximum number of iterations $K$ is predetermined, set $\alpha_k\equiv\frac{\theta}{\sqrt{K}}$ in \eqref{eq:ns-cvx-9} to obtain
$$\EE[\Phi(\tilde{\vx}^K)-\Phi(\vx^*)]\le\frac{D\theta}{\sqrt{K}}+\frac{\|\vx^*-\vx^1\|^2}{2\theta \sqrt{K}},$$
which completes the proof.
\end{proof}

Under the general convex setting, the rate  $O(1/\sqrt{k})$ is optimal for the SG method, and for strongly convex case, the rate $O(1/k)$ is optimal; see \cite{nemirovskici1983problem, nemirovski2009robust, lan2012optimal}. In the following theorem, we assume $\Phi(\vx)$ to be strongly convex and establish the rate $O(1/k)$. Hence, our algorithm  has the same orders of convergence rates as that of the SG method.

\begin{theorem}[Non-smooth strongly convex case]\label{thm:ns-strcvx}
Let $\{\vx^k\}$ be generated from Algorithm \ref{alg:bsg} with $\alpha_i^k=\alpha_k=\frac{\theta}{k}<\frac{1}{L},\forall i, k$. Under Assumptions \ref{assump1} through \ref{assump4}, if $F$ and $r_i$'s are convex, $\Phi$ is strongly convex with modulus $\mu>0$, namely,
$$\Phi(\lambda\vx+(1-\lambda)\vy)\le \lambda\Phi(\vx)+(1-\lambda)\Phi(\vy)-\frac{\mu}{2}\lambda(1-\lambda)\|\vx-\vy\|^2,\ \forall \vx,\vy, \forall \lambda\in[0,1],$$ and $\sigma=\sup_k\sigma_k<\infty$, then
\begin{equation}\label{eq:ns-strcvx-rate}
\EE [\|\vx^k-\vx^*\|^2] \le \frac{1}{k}\max\left\{\frac{2D\theta(1+\mu\theta)}{\mu},\|\vx^1-\vx^*\|^2\right\},
\end{equation}
where $D$ is defined in \eqref{eq:defD}.
\end{theorem}

\begin{proof}
When $\Phi$ is strongly convex with modulus $\mu>0$, it holds that
$$\Phi(\vx^{k+1})-\Phi(\vx^*)\ge\frac{\mu}{2}\|\vx^{k+1}-\vx^*\|^2,$$
which together with \eqref{eq:ns-cvx-7} implies
\begin{align}\label{eq:ns-scvx}
\EE\|\vx^{k+1}-\vx^*\|^2\le & \frac{\EE\|\vx^k-\vx^*\|^2}{1+\mu\alpha_k}+2D\frac{\alpha_k^2}{1+\mu\alpha_k}\cr
= & \frac{\EE\|\vx^k-\vx^*\|^2}{1+\mu\theta/k}+\frac{2D\theta^2}{k^2(1+\mu\theta/k)}\cr
\le & \left(1-\frac{\mu\theta}{k(1+\mu\theta)}\right)\EE\|\vx^k-\vx^*\|^2+\frac{2D\theta^2}{k^2}.
\end{align}
Using Lemma \ref{lem:subrate}, we immediately get the desired result and complete the proof.
\end{proof}

If the error $\sigma_k$ decreases fast, we can show an almost linear convergence result. We assume that either $\cX_i=\RR^{n_i}$ or $r_i=0$ for all $i$ since it is impossible in general to get linear convergence for subgradient method (see \cite{goffin1977convergence, boyd2003subgradient, nesterov2014convergent} and the references therein). For the convenience of our discussion, we let
\begin{equation}\label{eq:newr}
\hat{r}_i(\vx_i)=\left\{
\begin{array}{l}r_i(\vx_i),\ i\in\cI_1,\\[0.1cm]
\iota_{\cX_i}(\vx_i),\ i\in\cI_2.
\end{array}\right.
\end{equation}
This way, we can write the updates \eqref{eq:update} and \eqref{eq:update2} uniformly into \begin{equation}\label{eq:update3}
\vx_i^{k+1}=\argmin_{\vx_i}\langle \tilde{\vg}_i^k,\vx_i-\vx_i^k\rangle +\frac{1}{2\alpha_i^k}\|\vx_i-\vx_i^k\|^2+\hat{r}_i(\vx_i).
\end{equation}
Let
\begin{equation}\label{eq:newphi}\hat{\Phi}(\vx)=F(\vx)+\sum_{i=1}^s\hat{r}_i(\vx_i).
\end{equation}
Then problem \eqref{eq:main} becomes
$\min_\vx \hat{\Phi}(\vx)$
under the assumption that either $\cX_i=\RR^{n_i}$ or $r_i=0$ for all $i$. When $\Phi$ and thus $\hat{\Phi}$ is strongly convex  with modulus $\mu>0$, then from the discussion in section 2.2 of \cite{xu2013block}, it follows that
\begin{equation}\label{eq:kl}
\hat{\Phi}(\vx)-\hat{\Phi}(\vx^*)\le \frac{1}{\mu}\|\vg\|^2, \forall \vg\in\partial \hat{\Phi}(\vx).
\end{equation}
Using this result, we establish the following theorem.

\begin{theorem}\label{thm:strg-g}
Let $\{\vx^k\}$ be generated from Algorithm \ref{alg:bsg} with $\alpha_i^k=\alpha_k<\frac{2}{L},~\forall i, k$. Under Assumptions \ref{assump1} and \ref{assump2}, if either $\cX_i=\RR^{n_i}$ or $r_i=0$ for all $i$, $\Phi$ is strongly convex with constant $\mu$, and all $\hat{r}_i$'s in \eqref{eq:newr} are convex, then
\begin{equation}\label{eq:dec5}
\EE [\Phi(\vx^{k+1})-\Phi(\vx^*)]\le \frac{\gamma(\alpha_k)}{1+\gamma(\alpha_k)}\EE [\Phi(\vx^{k})-\Phi(\vx^*)]+\frac{s\nu(\alpha_k)}{1+\gamma(\alpha_k)}\sigma_k^2,
\end{equation}
where
\begin{equation}\label{eq:gamnu}\gamma(\alpha_k) = \frac{3L}{\mu^2}\big(\frac{1}{\alpha_k^2}+sL^2\big)\frac{1}{\frac{1}{2\alpha_k}-\frac{L}{4}},\quad \nu(\alpha_k)=\frac{\gamma(\alpha_k)\alpha_k}{2-L\alpha_k}+\frac{3L}{\mu^2}.
\end{equation}
\end{theorem}

\begin{remark}
This theorem does not require Assumptions \ref{assump3} or \ref{assump4}. Although it requires the strong convexity of $\Phi$, it does not require the convexity of $F$. 
\end{remark}

\begin{proof}
Note
$$\frac{1}{\alpha_k}(\vx_i^k-\vx_i^{k+1})-\tilde{\vg}_i^k\in\partial \hat{r}_i(\vx_i^{k+1}),$$
or equivalently
$$\frac{1}{\alpha_k}(\vx_i^k-\vx_i^{k+1})+\nabla_{\vx_i}F(\vx^{k+1})-\tilde{\vg}_i^k\in\partial \hat{r}_i(\vx_i^{k+1})+\nabla_{\vx_i}F(\vx^{k+1}).$$
Hence, from \eqref{eq:kl} it holds that
\begin{align*}
\hat{\Phi}(\vx^{k+1})-\hat{\Phi}(\vx^*)\le & \frac{1}{\mu}\sum_{i=1}^s\big\|\frac{1}{\alpha_k}(\vx_i^k-\vx_i^{k+1})+\nabla_{\vx_i}F(\vx^{k+1})-\tilde{\vg}_i^k\big\|^2\\
\le & \frac{3}{\mu}\sum_{i=1}^s\left(\big\|\frac{1}{\alpha_k}(\vx_i^k-\vx_i^{k+1})\big\|^2+\big\|\nabla_{\vx_i}F(\vx^{k+1})-\vg_i^k\big\|^2+\|\bm{\delta}_i^k\|^2\right).
\end{align*}
Noting
$\big\|\nabla_{\vx_i}F(\vx^{k+1})-\vg_i^k\big\|^2\le L^2\|\vx^{k+1}-\vx^k\|^2$, we have from the above inequality that
\begin{equation}\label{eq:k-kl}
\hat{\Phi}(\vx^{k+1})-\hat{\Phi}(\vx^*)\le\frac{3}{\mu}\big(\frac{1}{\alpha_k^2}+sL^2\big)\|\vx^{k+1}-\vx^k\|^2+\frac{3}{\mu}\sum_{i=1}^s\|\bm{\delta}_i^k\|^2.
\end{equation}
Letting $\vx=\vx^k$ in \eqref{eq:ns-cvx-3} gives
\begin{align}\label{eq:dec3}\hat{\Phi}(\vx^{k+1})-\hat{\Phi}(\vx^k)\le & -\sum_{i=1}^s\langle \bm{\delta}_i^k,\vx_i^{k+1}-\vx_i^k\rangle-\big(\frac{1}{\alpha_k}-\frac{L}{2}\big)\|\vx^{k+1}-\vx^k\|^2\cr
\le &\frac{\alpha_k}{2-L\alpha_k}\sum_{i=1}^s\|\bm{\delta}_i^k\|^2-\big(\frac{1}{2\alpha_k}-\frac{L}{4}\big)\|\vx^{k+1}-\vx^k\|^2,
\end{align}
which together with \eqref{eq:k-kl} implies
\begin{align}
&\hat{\Phi}(\vx^{k+1})-\hat{\Phi}(\vx^*)\cr
\le & \frac{L}{2}\|\vx^{k+1}-\vx^*\|^2\cr
\le &\frac{L}{\mu}\big(\hat{\Phi}(\vx^{k+1})-\hat{\Phi}(\vx^*)\big)\cr
\le & \frac{3L}{\mu^2}\big(\frac{1}{\alpha_k^2}+sL^2\big)\|\vx^{k+1}-\vx^k\|^2+\frac{3L}{\mu^2}\sum_{i=1}^s\|\bm{\delta}_i^k\|^2\cr
\le & \frac{3L}{\mu^2}\big(\frac{1}{\alpha_k^2}+sL^2\big)\frac{1}{\frac{1}{2\alpha_k}-\frac{L}{4}}\left(\Phi(\vx^k)-\Phi(\vx^{k+1})+\frac{\alpha_k}{2-L\alpha_k}\sum_{i=1}^s\|\bm{\delta}_i^k\|^2\right)\label{eq:dec4}\\
&+\frac{3L}{\mu^2}\sum_{i=1}^s\|\bm{\delta}_i^k\|^2,\nonumber
\end{align}
where the first inequality follows from the gradient Lipschitz continuity of $F$, the second one from the strong convexity of $\Phi$, the third one from \eqref{eq:k-kl}, and the fourth one from \eqref{eq:dec3}.
Taking expectation on both sides of \eqref{eq:dec4}, using \eqref{eq:var}, and arranging terms give
$$\EE[\hat{\Phi}(\vx^{k+1})-\hat{\Phi}(\vx^*)]\le \gamma(\alpha_k)\EE[\hat{\Phi}(\vx^k)-\hat{\Phi}(\vx^{k+1})]+s\nu(\alpha_k)\sigma_k^2,$$
which is  equivalent to the desired result \eqref{eq:dec5}. Noting $\Phi(\vx^k)=\hat{\Phi}(\vx^k)$ and $\Phi(\vx^*)=\hat{\Phi}(\vx^*)$ completes the proof.
\end{proof}

From \eqref{eq:dec5}, one can see that the convergence rate of Algorithm \ref{alg:bsg} depends on how fast $\sigma_k$ decreases. If $\sigma_k\equiv0$, i.e., the deterministic case, \eqref{eq:dec5} implies linear convergence. Generally, $\sigma_k$ does not vanish. However, one can decrease it by increasing the batch size $m_k$. This point will be discussed in Remark \ref{rm:var} below. Directly from \eqref{eq:dec5} and using the following lemma, we can get a convergence result  in Theorem \ref{thm:wkrate} below, whose proof follows \cite{friedlander2012hybrid}.

\begin{lemma}\label{lem:limmax}
Let $e_k=\max\{a_k, \eta^k\}$. If $a_{k+1}/a_k$ has a finite limit, then $e_{k+1}/e_k$ also has a finite limit, which equals $\eta$ or the limit of ${a_{k+1}}/{a_k}$, and the limit is no less than $\eta$.
\end{lemma}

\begin{theorem}\label{thm:wkrate}
Assume the assumptions in Theorem \ref{thm:strg-g}. Suppose $0<\underline{\alpha}=\inf_k\alpha_k\le\sup_k\alpha_k=\overline{\alpha}<\frac{2}{L}$. If ${\sigma_{k+1}^2}/{\sigma_k^2}$ has a finite limit,
then
\begin{equation}\label{eq:str-linear}\EE [\Phi(\vx^{k+1})-\Phi(\vx^*)] \le \eta^k(\Phi(\vx^1)-\Phi(\vx^*))+\cO(e_k),
\end{equation}
where $\eta=\frac{\gamma(\overline{\alpha})}{1+\gamma(\overline{\alpha})}$ and $e_k=\max\{\sigma_k^2,(\eta+\epsilon)^k\}$ for any $\epsilon>0$.
\end{theorem}

\begin{proof}
Note that $\gamma(\alpha)$ and $\nu(\alpha)$ defined in \eqref{eq:gamnu} are both increasing with respect to $\alpha\in (0, \frac{2}{L})$. Hence it follows from \eqref{eq:dec5} that
$$\EE [\Phi(\vx^{k+1})-\Phi(\vx^*)]\le \frac{\gamma(\overline{\alpha})}{1+\gamma(\overline{\alpha})}\EE [\Phi(\vx^{k})-\Phi(\vx^*)]+\frac{s\nu(\overline{\alpha})}{1+\gamma(\underline{\alpha})}\sigma_k^2.$$
Let $\tau = \frac{s\nu(\overline{\alpha})}{1+\gamma(\underline{\alpha})}$. Then using the above inequality recursively yields
$$\EE [\Phi(\vx^{k+1})-\Phi(\vx^*)]\le \eta^k(\Phi(\vx^1)-\Phi(\vx^*))+\tau\sum_{j=1}^k\eta^{k-j}\sigma_j^2.$$

From Lemma \ref{lem:limmax}, we have $\lim_{k\to\infty}e_{k+1}/e_k\ge \eta+\epsilon$, and thus there is a sufficiently large integer $K$ such that $e_{k+1}/e_k\ge \eta+\epsilon/2,\,\forall k\ge K$.
Let $\mu_k=\sum_{j=1}^k\eta^{k-j}\sigma_j^2$ and choose a sufficiently large number $B$ such that $\mu_k\le Be_k, \forall k\le K$ and $(B-1)\epsilon\ge 2\eta$. Note that such $B$ must exist since $K$ is finite and $\epsilon>0$. Suppose that for some $k\ge K$, it holds $\mu_k\le Be_k$. Then
$$\mu_{k+1}=\eta\mu_k+\sigma_{k+1}^2\le \eta Be_k+e_{k+1}\le \frac{\eta B}{\eta+\epsilon/2}e_{k+1}+e_{k+1}\le Be_{k+1},$$
where the last inequality uses the conditions $(B-1)\epsilon\ge 2\eta$ and $e_{k+1}\ge 0$.
Hence, by induction, we conclude $\mu_k\le Be_k, \forall k$.
This completes the proof.
\end{proof}

\begin{remark}
If $m_k=\cO(k)$ in Algorithm \ref{alg:bsg}, then according to Remark \ref{rm:var}, we have the convergence rate $\cO\big(\max\{((\eta+1)/2)^k, 1/k\}\big)$ by letting $\epsilon=\frac{1-\eta}{2}$ in Theorem \ref{thm:wkrate}. If $\sigma_k$ decreases linearly, then we have a linear convergence rate.
\end{remark}

\subsection{Nonconvex case}
In this section, we do not assume any convexity of $F$ or $r_i$'s. First, we analyze Algorithm \ref{alg:bsg} for the smooth case with $r_i=0$ and $\cX_i=\RR^{n_i}$ for all $i$. Then, we impose a more restrictive condition on $\sigma_k$ and analyze the algorithm for the unconstrained nonsmooth and constrained smooth cases. We start our analysis with the following lemma, which can be found in Lemma A.5 of \cite{mairal2013mm} and
Proposition 1.2.4 of \cite{Bertsekas-NLP}.
\begin{lemma}\label{lem:seq}
For two nonnegative scalar sequences $\{a_k\}$ and $\{b_k\}$, if $\sum_{k=1}^\infty a_k=+\infty$ and $\sum_{k=1}^\infty a_kb_k<+\infty$, then
$$\liminf_{k\to\infty} b_k=0.$$
Furthermore, if $|b_{k+1}-b_k|\le B\cdot a_k$ for some constant $B>0$, then
$$\lim_{k\to\infty}b_k=0.$$
\end{lemma}

\begin{theorem}\label{thm:convg1}
Let $\{\vx^k\}$ be generated from Algorithm \ref{alg:bsg} with $r_i=0$ and $\cX_i=\RR^{n_i}$ for all $i$ and with $\alpha_i^k=c_i^k\beta_k$, where $c_i^k, \beta_k$ are positive scalars such that $\beta_k\ge \beta_{k+1},\forall k$,
\begin{equation}\label{eq:bdc}
0<\inf_k c_i^k\le \sup_k c_i^k<\infty, \forall i,
\end{equation} and
\begin{equation}\label{eq:beta}
\sum_{k=1}^\infty\beta_k=+\infty,\qquad \sum_{k=1}^\infty\beta_k^2 < +\infty.
\end{equation}
Under Assumptions \ref{assump1} through \ref{assump3}, if $\sigma=\sup_k\sigma_k<\infty$, then
\begin{equation}\label{eq:convgrad}
\lim_{k\to\infty}\EE \|\nabla \Phi(\vx^k)\|=\lim_{k\to\infty}\EE \|\nabla F(\vx^k)\| = 0.
\end{equation}
\end{theorem}

\begin{proof}
From the Lipschitz continuity of $\nabla_{\vx_i}F$, it holds that
\begin{align}
& F(\vx_{\le i}^{k+1},\vx_{>i}^k)-F(\vx_{< i}^{k+1},\vx_{\ge i}^k)\cr
\le& \langle \vg_i^k,\vx_i^{k+1}-\vx_i^k\rangle +\frac{L}{2}\|\vx_i^{k+1}-\vx_i^k\|^2\cr
= &-\alpha_i^k\langle \vg_i^k,\tilde{\vg}_i^k\rangle +\frac{L}{2}(\alpha_i^k)^2\|\tilde{\vg}_i^k\|^2\cr
=& -\big(\alpha_i^k-\frac{L}{2}(\alpha_i^k)^2\big)\|\vg_i^k\|^2+\frac{L}{2}(\alpha_i^k)^2\|\bm{\delta}_i^k\|^2-\big(\alpha_i^k-L(\alpha_i^k)^2\big)\langle \vg_i^k,\bm{\delta}_i^k\rangle\cr
=&-\big(\alpha_i^k-\frac{L}{2}(\alpha_i^k)^2\big)\|\vg_i^k\|^2+\frac{L}{2}(\alpha_i^k)^2\|\bm{\delta}_i^k\|^2-\big(\alpha_i^k-L(\alpha_i^k)^2\big)\big(\langle \vg_i^k-\nabla_{\vx_i}F(\vx^k),\bm{\delta}_i^k\rangle+\langle\nabla_{\vx_i}F(\vx^k),\bm{\delta}_i^k\rangle\big)\cr
\le& -\big(\alpha_i^k-\frac{L}{2}(\alpha_i^k)^2\big)\|\vg_i^k\|^2+\frac{L}{2}(\alpha_i^k)^2\|\bm{\delta}_i^k\|^2-\big(\alpha_i^k-L(\alpha_i^k)^2\big)\langle \nabla_{\vx_i}F(\vx^k),\bm{\delta}_i^k\rangle\label{eq:eq1}\\
&+L\big(\alpha_i^k+L(\alpha_i^k)^2\big)\max_j\alpha_j^k\big(\|\bm{\delta}_i^k\|^2+\sum_{j=1}^s(\|\vg_j^k\|^2+\|\bm{\delta}_j^k\|^2)\big),\nonumber
\end{align}
where the last inequality follows from the following argument:
\begin{align*}
&-\big(\alpha_i^k-L(\alpha_i^k)^2\big)\langle \vg_i^k-\nabla_{\vx_i}F(\vx^k),\bm{\delta}_i^k\rangle\\
\le & \big|\alpha_i^k-L(\alpha_i^k)^2\big|\|\bm{\delta}_i^k\|\|\vg_i^k-\nabla_{\vx_i}F(\vx^k)\|\\
\le & L|\alpha_i^k-L(\alpha_i^k)^2|\|\bm{\delta}_i^k\|\|\vx_{<i}^{k+1}-\vx_{<i}^k\|\qquad(\text{from gradient Lipschitz continuity of }F)\\
\le & L|\alpha_i^k-L(\alpha_i^k)^2|\|\bm{\delta}_i^k\|\sqrt{\sum_{j=1}^s\|\alpha_j^k\tilde{\vg}_j^k\|^2}\qquad(\text{note }\vx_j^{k+1}=\vx_j^k-\alpha_j^k\tilde{\vg}_j^k)\\
\le& L\big(\alpha_i^k+L(\alpha_i^k)^2\big)\max_j\alpha_j^k\big(\|\bm{\delta}_i^k\|^2+\sum_{j=1}^s(\|\vg_j^k\|^2+\|\bm{\delta}_j^k\|^2)\big).\qquad(\text{from triangle inequality})
\end{align*}
Summing \eqref{eq:eq1} over $i$ and using \eqref{eq:var} give
\begin{align}
F(\vx^{k+1})-F(\vx^k)\le & -\sum_{i=1}^s \big(\alpha_i^k-\frac{L}{2}(\alpha_i^k)^2\big)\|\vg_i^k\|^2 -\sum_{i=1}^s \big(\alpha_i^k-L(\alpha_i^k)^2\big)\langle \nabla_{\vx_i}F(\vx^k),\bm{\delta}_i^k\rangle\label{eq:ineqsum}\\
&\hspace{-2cm}+ \sum_{i=1}^s \left(\frac{L}{2}(\alpha_i^k)^2\|\bm{\delta}_i^k\|^2+L\big(\alpha_i^k+L(\alpha_i^k)^2\big)\max_j\alpha_j^k\big(\|\bm{\delta}_i^k\|^2+\sum_{j=1}^s(\|\vg_j^k\|^2+\|\bm{\delta}_j^k\|^2)\big)\right).\nonumber
\end{align}
Note that since $\vx^k$ is independent of $\Xi_k$. Hence, from Lemma \ref{lem:indp}, we have
\begin{align*}\EE \langle \nabla_{\vx_i}F(\vx^k),\bm{\delta}_i^k\rangle\le A\cdot\max_j\alpha_j^k\cdot\EE\|\nabla_{\vx_i}F(\vx^k)\|
\le M_\rho A\cdot\max_j\alpha_j^k,
\end{align*}
where $M_\rho$ is defined in \eqref{eq:defM}.
Taking expectation over \eqref{eq:ineqsum}, we have
\begin{align}
&\EE F(\vx^{k+1})-\EE F(\vx^k)\cr
\le & -\sum_{i=1}^s \big(\alpha_i^k-\frac{L}{2}(\alpha_i^k)^2\big)\EE\|\vg_i^k\|^2+ \sum_{i=1}^s M_\rho A(\alpha_i^k+L(\alpha_i^k)^2)\max_j\alpha_j^k\cr
& +\sum_{i=1}^s \left(\frac{L}{2}(\alpha_i^k)^2\sigma^2+L\big(\alpha_i^k+L(\alpha_i^k)^2\big)\max_j\alpha_j^k
\big((s+1)\sigma^2+sM_\rho^2\big)\right)\cr
\le & -\sum_{i=1}^s (c\beta_k-\frac{LC^2}{2}\beta_k^2)\EE\|\vg_i^k\|^2+ \sum_{i=1}^s M_\rho AC\beta_k(C\beta_k+LC^2\beta_k^2) \label{eq:dec}\\
&+\sum_{i=1}^s \left(\frac{LC^2}{2}\beta_k^2\sigma^2+L\beta_k(C\beta_k+LC^2\beta_k^2)\big((s+1)\sigma^2+sM_\rho^2\big)\right),\nonumber
\end{align}
where
\begin{equation}\label{eq:defC}
c=\min_i\inf_k c_i^k,\qquad C=\max_i\sup_k c_i^k.
\end{equation}

Note that $F$ is lower bounded, $\EE\|\vg_i^k\|^2\le M_\rho^2$, and $0<c\le C<\infty$.
Summing \eqref{eq:dec} over $k$ and using \eqref{eq:beta}, we have
$$\sum_{k=1}^\infty \beta_k\EE\|\vg_i^k\|^2 < \infty, \forall i.$$
Furthermore,
\begin{align}
&\left|\EE\|\vg_i^{k+1}\|^2-\EE\|\vg_i^{k}\|^2\right|\cr
\le & \EE\big[\|\vg_i^{k+1}+\vg_i^{k}\|\cdot\|\vg_i^{k+1}-\vg_i^{k}\|\big]\cr
\le & 2LM_\rho\EE\big\|(\vx_{j<i}^{k+2},\vx_{j\ge i}^{k+1})-(\vx_{j<i}^{k+1},\vx_{j\ge i}^{k})\big\|\qquad(\text{from Remark \ref{rm:bditer} and gradient Lipschitz continuity of }F)\cr
= & 2LM_\rho\EE\sqrt{\sum_{j<i}\|\alpha_j^{k+1}\tilde{\vg}_j^{k+1}\|^2+\sum_{j\ge i}\|\alpha_j^{k}\tilde{\vg}_j^{k}\|^2}\qquad(\text{note }\vx_j^{k+1}=\vx_j^k-\alpha_j^k\tilde{\vg}_j^k)\cr
\le & 2LM_\rho C\beta_k\EE\sqrt{\sum_{j<i}\|\tilde{\vg}_j^{k+1}\|^2+\sum_{j\ge i}\|\tilde{\vg}_j^{k}\|^2}\qquad(\text{from definition of }\alpha_i^k\text{ and the monotonicity of }\beta_k)\cr
\le & 2LM_\rho C\beta_k\sqrt{\EE\big[\sum_{j<i}\|\tilde{\vg}_j^{k+1}\|^2+\sum_{j\ge i}\|\tilde{\vg}_j^{k}\|^2\big]}\qquad(\text{from Jenson's inequality})\cr
\le & 2LM_\rho C\beta_k\sqrt{2s(M_\rho^2+\sigma^2)}.\qquad(\text{from Jenson's inequality})\label{eq:gdiff}
\end{align}
According to Lemma \ref{lem:seq}, we have $\EE \|\vg_i^k\|^2\to 0,~\forall i$, as $k\to\infty$ and thus $\EE \|\vg_i^k\|\to 0,~\forall i$, as $k\to\infty$ by Jensen's inequality. Hence,
\begin{align*}\EE\|\nabla_{\vx_i}F(\vx^k)\|\le & \EE\|\nabla_{\vx_i}F(\vx^k)-\vg_i^k\|+\EE\|\vg_i^k\|\\
\le & L\cdot\EE\|\vx_{<i}^{k+1}-\vx_{<i}^k\|+\EE\|\vg_i^k\|\\
 \le & LC\sqrt{2s(M_\rho^2+\sigma^2)}\beta_k+\EE\|\vg_i^k\| \to 0,\ \forall i, \text{ as }k\to\infty,
\end{align*}
where the last inequality is obtained following the same argument for \eqref{eq:gdiff}.
Therefore, the desired result is obtained.
\end{proof}

\begin{remark}
The above proof only needs the boundedness of $\EE\|\vg_i^k\|^2,~\forall i, k$, instead of the stronger Assumption \ref{assump3}.
\end{remark}

We next analyze Algorithm \ref{alg:bsg} for the non-smooth case with the notation in \eqref{eq:newr} through \eqref{eq:newphi}.

\begin{theorem}\label{thm:convg2}
Let $\{\vx^k\}$ be generated from Algorithm \ref{alg:bsg} with $\alpha_i^k$'s being the same as those in Theorem \ref{thm:convg1}. Under Assumptions \ref{assump1} through \ref{assump4}, if either  $\cX_i=\RR^{n_i}$ or $r_i=0$ for all $i$, and
\begin{equation}\label{eq:beta2}\sum_{k=1}^\infty\beta_k\sigma_k^2<\infty,
\end{equation}
then there exists an index subsequence $\cK$ such that
\begin{equation}\label{eq:subconvg}
\lim_{\substack{k\to\infty\\ k\in\cK}}\EE \big[\dist\big(\vzero,\partial \hat{\Phi}(\vx^k)\big)\big] = 0,
\end{equation}
where $\dist(\vy,\cX)=\min_{\vx\in\cX}\|\vx-\vy\|$.
\end{theorem}

\begin{proof}
From the optimality of $\vx_i^{k+1}$ for \eqref{eq:update3}, it holds that
$$ \hat{r}_i(\vx_i^{k+1})- \hat{r}_i(\vx_i^k)\le -\langle\tilde{\vg}_i^k,\vx_i^{k+1}-\vx_i^{k}\rangle-\frac{1}{2\alpha_i^k}\|\vx_i^{k+1}-\vx_i^{k}\|^2.$$
In addition, from the Lipschitz continuity condition \eqref{eq:lipobj}, it follows that
$$ F(\vx_{\le i}^{k+1},\vx_{>i}^k)-F(\vx_{< i}^{k+1},\vx_{\ge i}^k)
\le \langle \vg_i^k,\vx_i^{k+1}-\vx_i^k\rangle +\frac{L}{2}\|\vx_i^{k+1}-\vx_i^k\|^2.$$
From \eqref{eq:bdc} and \eqref{eq:beta}, we have $\alpha_i^k\to 0, \forall i$ as $k\to\infty$ and can thus take sufficiently large $\tilde{k}$ such that $\alpha_i^k<\frac{1}{2L},\forall i, k\ge \tilde{k}$. Summing up the above two inequalities and assuming $k\ge\tilde{k}$, we have
\begin{align}
&\hat{\Phi}(\vx_{\le i}^{k+1},\vx_{>i}^k)-\hat{\Phi}(\vx_{< i}^{k+1},\vx_{\ge i}^k)\cr
\le &-\langle \bm{\delta}_i^k,\vx_i^{k+1}-\vx_i^k\rangle -(\frac{1}{2\alpha_i^k}-\frac{L}{2})\|\vx_i^{k+1}-\vx_i^k\|^2\label{eq:eq7}\\
\le &\ \frac{\alpha_i^k}{1-L\alpha_i^k}\|\bm{\delta}_i^k\|^2-\frac{1}{2}(\frac{1}{2\alpha_i^k}-\frac{L}{2})\|\vx_i^{k+1}-\vx_i^k\|^2\label{eq:eq6}\\
=&\ \frac{\alpha_i^k}{1-L\alpha_i^k}\|\bm{\delta}_i^k\|^2-\big(\frac{\alpha_i^k}{4}-\frac{L}{4}(\alpha_i^k)^2\big)\|\tilde{\vh}_i^k\|^2\nonumber
\end{align}
where we have used the Young's inequality in the second inequality. Summing  the above inequality over $i$ yields
\begin{equation}\label{eq:dec2}
\hat{\Phi}(\vx^{k+1})-\hat{\Phi}(\vx^k) \le \sum_{i=1}^s\frac{\alpha_i^k}{1-L\alpha_i^k}\|\bm{\delta}_i^k\|^2-\sum_{i=1}^s\big(\frac{\alpha_i^k}{4}-\frac{L}{4}(\alpha_i^k)^2\big)\|\tilde{\vh}_i^k\|^2.
\end{equation}
Taking expectation on both sides of the above inequality gives
\begin{align}\EE \hat{\Phi}(\vx^{k+1})-\EE \hat{\Phi}(\vx^k) \le & \sum_{i=1}^s\frac{\alpha_i^k}{1-L\alpha_i^k}\sigma_k^2-\sum_{i=1}^s\big(\frac{\alpha_i^k}{4}-\frac{L}{4}(\alpha_i^k)^2\big)\EE \|\tilde{\vh}_i^k\|^2\cr
\le &\sum_{i=1}^s 2C\beta_k\sigma_k^2-\sum_{i=1}^s\big(\frac{c}{4}\beta_k-\frac{L}{4}C^2\beta_k^2\big)\EE \|\tilde{\vh}_i^k\|^2,\label{eq:eq3}
\end{align}
where in the second inequality, we have assumed $\alpha_i^k<\frac{1}{2L}$ by taking $k\ge\tilde{k}$, and $c$ and $C$ are defined in \eqref{eq:defC}.
Note that from Lemma \ref{lem:bdh}, we have that $\EE\|\tilde{\vh}_i^k\|^2$ is bounded for all $i$ and $k$. In addition, by \eqref{eq:beta} and \eqref{eq:beta2} and recalling that $\Phi$, and thus $\hat{\Phi}$, is lower bounded, we have from \eqref{eq:eq3} that
$$\sum_{k=0}^\infty\beta_k\EE\|\tilde{\vh}_i^k\|^2<\infty, \forall i.$$
Hence, from Lemma \ref{lem:seq}, there must be a subsequence $\{\EE \|\tilde{\vh}_i^k\|^2\}_{k\in\cK}$ converging to \emph{zero} for all $i$. 

Furthermore, note that
$$\tilde{\vh}_i^k=\frac{1}{\alpha_i^k}(\vx_i^k-\vx_i^{k+1})\in \tilde{\vg}_i^k+\partial \hat{r}_i(\vx_i^{k+1}), \forall i,$$
which is equivalent to
$$\tilde{\vh}_i^k+\nabla_{\vx_i}F(\vx^{k+1})- \tilde{\vg}_i^k\in\nabla_{\vx_i}F(\vx^{k+1})+\partial \hat{r}_i(\vx_i^{k+1}),\forall i.$$
Hence,
\begin{align*}
\dist(\vzero,\partial \hat{\Phi}(\vx^{k+1}))^2\le&\sum_{i=1}^s\|\tilde{\vh}_i^k+\nabla_{\vx_i}F(\vx^{k+1})- \tilde{\vg}_i^k\|^2\\
\le&2\sum_{i=1}^s\big(\|\tilde{\vh}_i^k\|^2+\|\nabla_{\vx_i}F(\vx^{k+1})- \tilde{\vg}_i^k\|^2\big)\\
\le&4\sum_{i=1}^s\left(\|\tilde{\vh}_i^k\|^2+\|\nabla_{\vx_i}F(\vx^{k+1})- {\vg}_i^k\|^2+\|\bm{\delta}_i^k\|^2\right),
\end{align*}
and thus taking expectation over the last inequality gives
$$\EE[\dist(\vzero,\partial \hat{\Phi}(\vx^{k+1}))^2]\le4\sum_{i=1}^s\left(\EE\|\tilde{\vh}_i^k\|^2+\EE\|\nabla_{\vx_i}F(\vx^{k+1})- {\vg}_i^k\|^2+\EE\|\bm{\delta}_i^k\|^2\right).$$
Following the  argument same as that at the end of the proof of Theorem \ref{thm:convg1}, one can easily show that $\EE\|\nabla_{\vx_i}F(\vx^{k+1})- {\vg}_i^k\|^2\to 0,~\forall i$, as $k\to\infty$. Taking another subsequence of $\cK$ if necessary, we have $\EE\|\bm{\delta}_i^k\|^2\le\sigma_k^2\to 0,~ \forall i,$ as $\cK\ni k\to\infty$ from \eqref{eq:beta2}.
Hence, the right-hand side of the last inequality converges to \emph{zero} as $\cK \ni k\to\infty$. Applying Jensen's inequality completes the proof.
\end{proof}


From \eqref{eq:dec2} in the above proof, if $\sigma_k$ decreases in a faster way, we can show that \eqref{eq:subconvg} holds for $\cK$ being the whole index sequence. This result is summarized below.

\begin{corollary}
Let $\{\vx^k\}$ be generated from Algorithm \ref{alg:bsg} with $\alpha_i^k$'s satisfying $0<\inf_k\alpha_i^k\le\sup_k\alpha_i^k<\frac{1}{L},~\forall i$. Under Assumptions \ref{assump1} through \ref{assump4}, if either $\cX_i=\RR^{n_i}$ or $r_i=0$ for all $i$ and $\sum_{k=1}^\infty \sigma_k^2<\infty$, then
$$\lim_{k\to\infty}\EE \big[\dist\big(\vzero,\partial \hat{\Phi}(\vx^k)\big)\big] = 0,$$
which says that the optimality condition for \eqref{eq:main} asymptotically holds in expectation.
\end{corollary}

\begin{proof}
Under the conditions that $0<\inf_k\alpha_i^k\le\sup_k\alpha_i^k<\frac{1}{L},\forall i$ and $\sum_{k=1}^\infty \sigma_k^2<\infty$, it follows from \eqref{eq:dec2} that the subsequence $\cK$ in the proof of Theorem \ref{thm:convg2} can be taken as the whole sequence. Noting $\sigma_k\to 0$ as $k\to\infty$, we immediately get the desired result from Theorem \ref{thm:convg2}.
\end{proof}

\begin{remark}\label{rm:var}
One way to make $\sum_{k=1}^\infty\sigma_k^2<\infty$ is to asymptotically increase $m_k$ at a sufficiently fast rate. Let
$$\bm{\delta}_{i,\ell}^k=\nabla_{\vx_i}f(\vx_{<i}^{k+1},\vx_{\ge i}^k;\xi_{k,\ell})-\vg_i^k.$$
Then, $\bm{\delta}_i^k=\frac{1}{m_k}\sum_{\ell=1}^{m_k}\bm{\delta}_{i,\ell}^k$. As in Assumption \ref{assump1}, assume $\|\EE[\bm{\delta}_{i,\ell}^k|\bm{\Xi}_{[k-1]}]\|\le A\cdot\max_j\alpha_j^k$ and also $\EE\|\bm{\delta}_{i,\ell}^k\|^2\le \sigma^2$ for some constants $A$ and $\sigma$. Then following the proof of \cite{ghadimi2013mini} on page 11, one can show that $\EE\|\bm{\delta}_i^k\|^2\le \cO(1/m_k)$. Hence, taking $m_k=\lceil k^{1+\epsilon}\rceil$ for any $\epsilon>0$ guarantees $\sum_{k=1}^\infty\sigma_k^2<\infty$, where $\lceil a\rceil$ denotes the smallest integer that is no less than $a$.
\end{remark}

\section{Numerical experiments}\label{sec:numerical}

In this section, we report the simulation results of Algorithm \ref{alg:bsg}, dubbed as BSG, on both convex and nonconvex problems to demonstrate its advantages. The tested convex problems include stochastic least squares \eqref{eq:ls} and logistic regression \eqref{eq:lreg}. The tested nonconvex problems include low-rank tensor recovery \eqref{eq:lrtc} and bilinear logistic regression \eqref{eq:blr}.

\subsection{Parameter settings}
For convex problems \eqref{eq:ls} and \eqref{eq:lreg}, we let $(\pi_1^k,\ldots,\pi_s^k)$ be a \emph{random shuffling} of $(1,\ldots,s)$ and set the stepsize of BSG to $\alpha_i^k=\min(\frac{\theta}{\sqrt{k}},\frac{1}{L_i^k}),~\forall i,k$, where the value of $\theta$ was specified in each test, and $L_i^k$ was the Lipschitz constant of
\begin{equation}\label{eq:appf}\frac{1}{m_k}\sum_{\ell=1}^{m_k}\nabla_{\vx_{\pi_i^k}} f(\vx_{\pi_{<i}^k}^{k+1},\vx_i,\vx_{\pi_{>i}^k}^k;\xi_{k,\ell})
\end{equation} with respect to $\vx_i$. In addition, each $\xi_{k,\ell}$ was generated uniformly at random, and $m_k$ was set the same for all $k$ and specified below. We treat each coordinate as a block, i.e., $s=n$. Within each iteration, although BSG requires computing block partial gradient $s$ times, we need very little extra computation (with complexity $\cO(1)$) to get the new partial gradient from the previous one due to cyclic update and thus greatly save the computing time.

For nonconvex problems \eqref{eq:lrtc} and \eqref{eq:blr}, we treat each factor matrix as a block, i.e., $s=3$ for \eqref{eq:lrtc} and $s=2$ for \eqref{eq:blr}. We used the fixed updated order by letting $\pi_i^k=i,\,\forall i,k$, and set $\alpha_i^k=\min(\frac{\theta}{\sqrt{k}\log k},\frac{1}{L_i^k})$ for smooth nonconvex problems, i.e., problems \eqref{eq:lrtc} with $\lambda=0$ and \eqref{eq:blr}, and  set $\alpha_i^k=\frac{1}{L_i^k}$ for non-smooth nonconvex ones, i.e., problem \eqref{eq:lrtc} with $\lambda>0$, where $L_i^k$ was the Lipschitz constant of \eqref{eq:appf}. In addition, for smooth nonconvex cases, $m_k$ was set the same for all $k$, and for non-smooth nonconvex case, we asymptotically increased it by $m_k=m_1+\lceil\frac{k-1}{10}\rceil,~\forall k$.

We compared BSG with the SG (stochastic gradient) method and  the SBMD (stochastic block mirror descent) method \cite{dang2013stochastic} on problems \eqref{eq:ls} and \eqref{eq:lreg}. Stepsize $\alpha_k$ for SG and SBMD was set in the same way as $\alpha_i^k$ in the above. Specifically, $\alpha_k=\min(\frac{\theta}{\sqrt{k}},\frac{1}{L_k})$, where $L_k$ was the Lipschitz constant of $\frac{1}{m_k}\sum_{\ell=1}^{m_k}\nabla_{\vx} f(\vx;\xi_{k,\ell})$ for SG and $\frac{1}{m_k}\sum_{\ell=1}^{m_k}\nabla_{\vx_{i_k}} f(\vx_{\neq i_k}^{k},\vx_{i_k};\xi_{k,\ell})$ with respect to the selected block $\vx_{i_k}$ for SBMD. On solving \eqref{eq:lrtc} and \eqref{eq:blr}, we compared BSG with the BCGD (block coordinate gradient descent) method \cite{TsengYun2009}, whose stepsize $\lambda_i^k$ was taken as the reciprocal of the Lipschitz constant of $\nabla_{\vx_i}F(\vx_{<i}^{k+1},\vx_i,\vx_{>i}^k)$ with respect to $\vx_i$ for all $i$ and $k$. Throughout our tests, all compared algorithms were supplied with the same randomly generated starting point.

\subsection{Stochastic least squares}
We tested BSG, SG, and SBMD on the problem:
\begin{equation}\label{eq:ls}
\min_\vx \EE_{\va,b} \frac{1}{2}(\va^\top\vx-b)^2,
\end{equation}
where $\va$ and $b$ were random variables. In this test, entries of $\va$ independently followed the standard Gaussian distribution, and $b=\va^\top\hat{\vx}+\eta$ where $\eta$ was independent of $\va$ and followed the Gaussian distribution $\cN(0,0.01)$, and $\hat{\vx}\in\RR^{200}$ was a deterministic vector. It is easy to show that $\hat{\vx}$ was the solution to \eqref{eq:ls}, and the optimal objective value was $0.005$. We first generated a Gaussian random vector $\hat{\vx}$. Then we generated $N$ samples of $\va$ and $b$ according to their distributions, one at a time, and for each sample, we performed one update of the three algorithms, i.e., $m_k=1$ in \eqref{eq:appf}. All three algorithms started from the same Gaussian randomly generated point and used $\theta=0.1$. To compare their solutions, we generated another 100,000 samples $(\va,b)$ following the same distribution and calculated the empirical loss. The entire process was repeated 100 times independently, and average empirical losses were shown in Table \ref{table:ls} for different $N$'s. ``SBMD-t'' denotes SBMD algorithm that independently selected $t$ coordinates at each iteration. Since ``SBMD-200'' becomes SG method, its results were identical to those of SG and thus not reported. From the results, we see that SBMD performed consistently better by updating more coordinates at each iteration. BSG was better than SG except for $N=4000$. Note that SBMD only renewed partial coordinates at each update and thus took less computing time.

\begin{table}\caption{Objective values of BSG, SG, and SBMD on solving stochastic least squares \eqref{eq:ls} with data following the Gaussian distribution. \textbf{Bold numbers} are best.}\label{table:ls}
\begin{center}
{\footnotesize
\begin{tabular}{|c||c|c|c|c|c|}\hline
$N$ (Total Samples) & BSG & SG & SBMD-10 & SBMD-50 &SBMD-100\\\hline
4000 & 6.45e-3 & \textbf{6.03e-3} & 67.49 & 4.79 & 1.03e-1 \\
6000 & \textbf{5.69e-3} & 5.79e-3 & 53.84 & 1.43 & 1.43e-2 \\
8000 & \textbf{5.57e-3} & 5.65e-3 & 42.98 & 4.92e-1 & 6.70e-3 \\
10000 & \textbf{5.53e-3} & 5.58e-3 & 35.71 & 2.09e-1 & 5.74e-3\\\hline
\end{tabular}}
\end{center}
\end{table}

\subsection{Logistic regression}\label{sec:logreg}
We tested BSG, SG, and SBMD on the  problem:
\begin{equation}\label{eq:lreg}
\min_{\vw,b} \frac{1}{N}\sum_{\ell=1}^N\log(1+\exp(-y_\ell(\vx_\ell^\top\vw+b)).
\end{equation}
First, we compared the three algorithms on synthetic data. We randomly generated $N=2000$ samples of dimension 200 with  half of them belonging to the ``$+1$'' class and the other half to the ``$-1$'' class. Each sample in the positive class has components independently sampled from the Gaussian distribution $\cN(5, 1)$ and those in the negative class from $\cN(-5,1)$. We ran BSG, SG, and SBMD each to 50 epochs or 2 seconds, where one epoch was equivalent to going through all samples once. At each iteration of the algorithms, we uniformly randomly selected one sample, i.e., $m_k=1$ in \eqref{eq:appf}. Three different values of $\theta$ were tested. Figure \ref{fig:lreg-rand} plots the gap between the optimal objective value and those given by different algorithms for solving \eqref{eq:lreg}, where the optimal objective value was accurately obtained by running FISTA \cite{BeckTeboulle2009} to 5,000 iteration. From the figure, we see that when $\theta$ was small, SBMD performed consistently better if more coordinates were updated at each iteration, and when $\theta$ was large, its performance was almost irrelevant to the numbers of updated coordinates. In addition, the proposed BSG method performed the best, and it reached a much lower objective within the same number of epochs or the same amount of running time, especially when a large $\theta$ was used. With respect to running time, the worse performance of SBMD-1 compared to BSG is possibly because SBMD-1 used all coordinates to evaluate every partial gradient (with complexity $\cO(n)$) while BSG used all coordinates only for the first partial gradient and then just the renewed coordinate for all other partial gradients (each with complexity $\cO(1)$) due to cyclic update.

\begin{figure}\caption{Objective values of BSG, SG, and SBMD for solving logistic regression \eqref{eq:lreg} on Gaussian randomly generated samples.}\label{fig:lreg-rand}
\centering
\begin{tabular}{ccc}
$\theta=0.1$ & $\theta = 1$  & $\theta = 10$\\[-0.2cm]
\includegraphics[width = 0.31\textwidth]{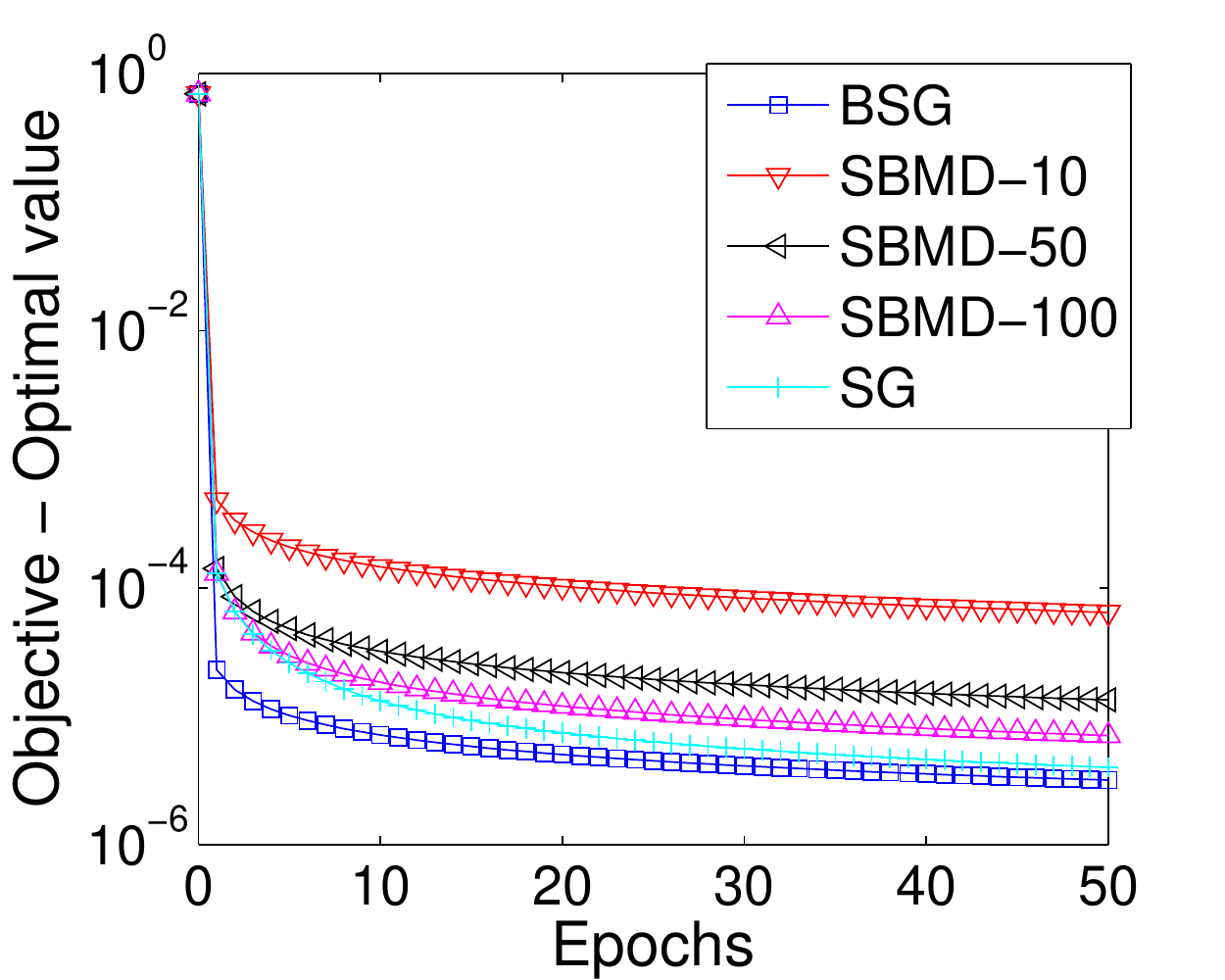}&
\includegraphics[width = 0.31\textwidth]{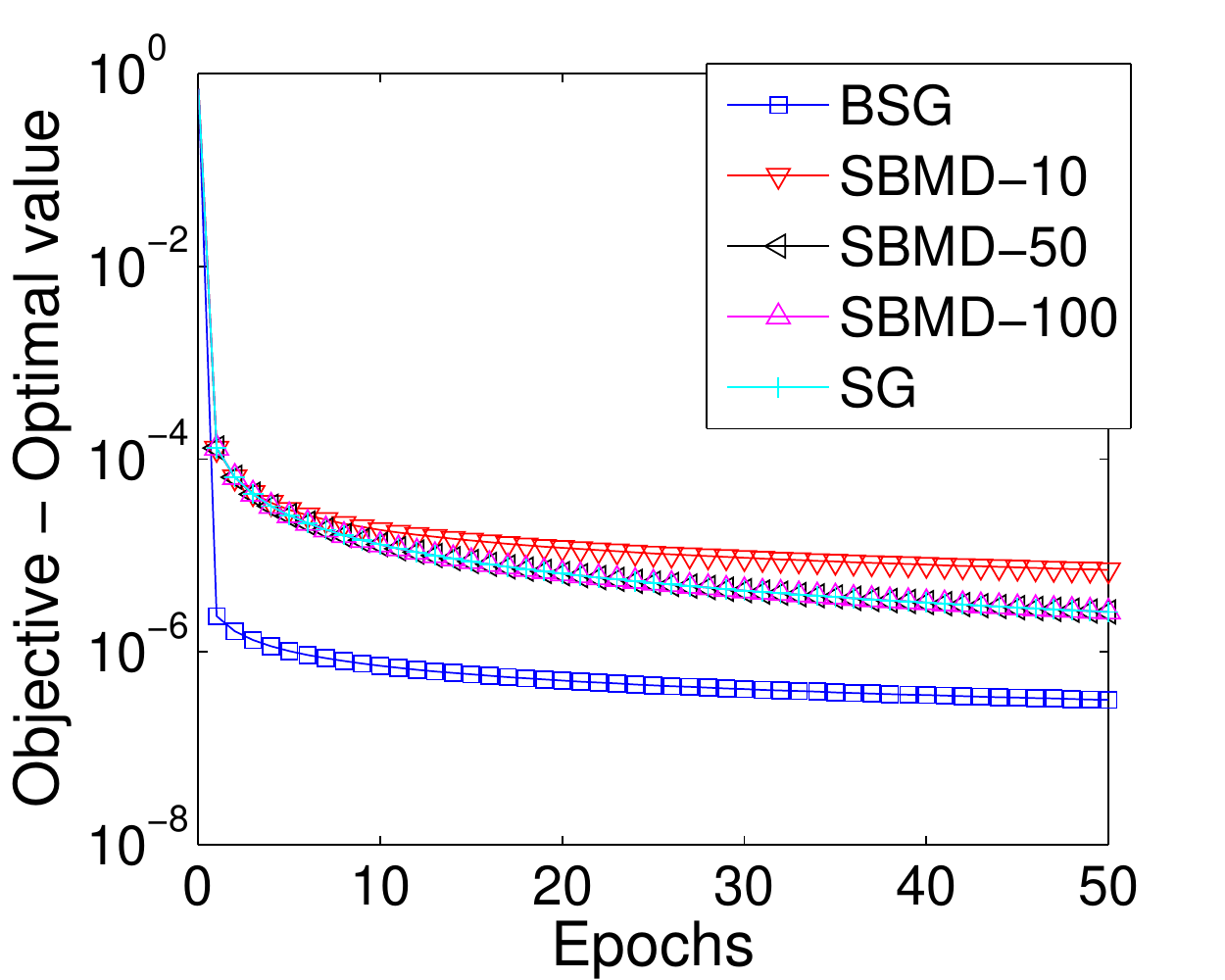}&
\includegraphics[width = 0.31\textwidth]{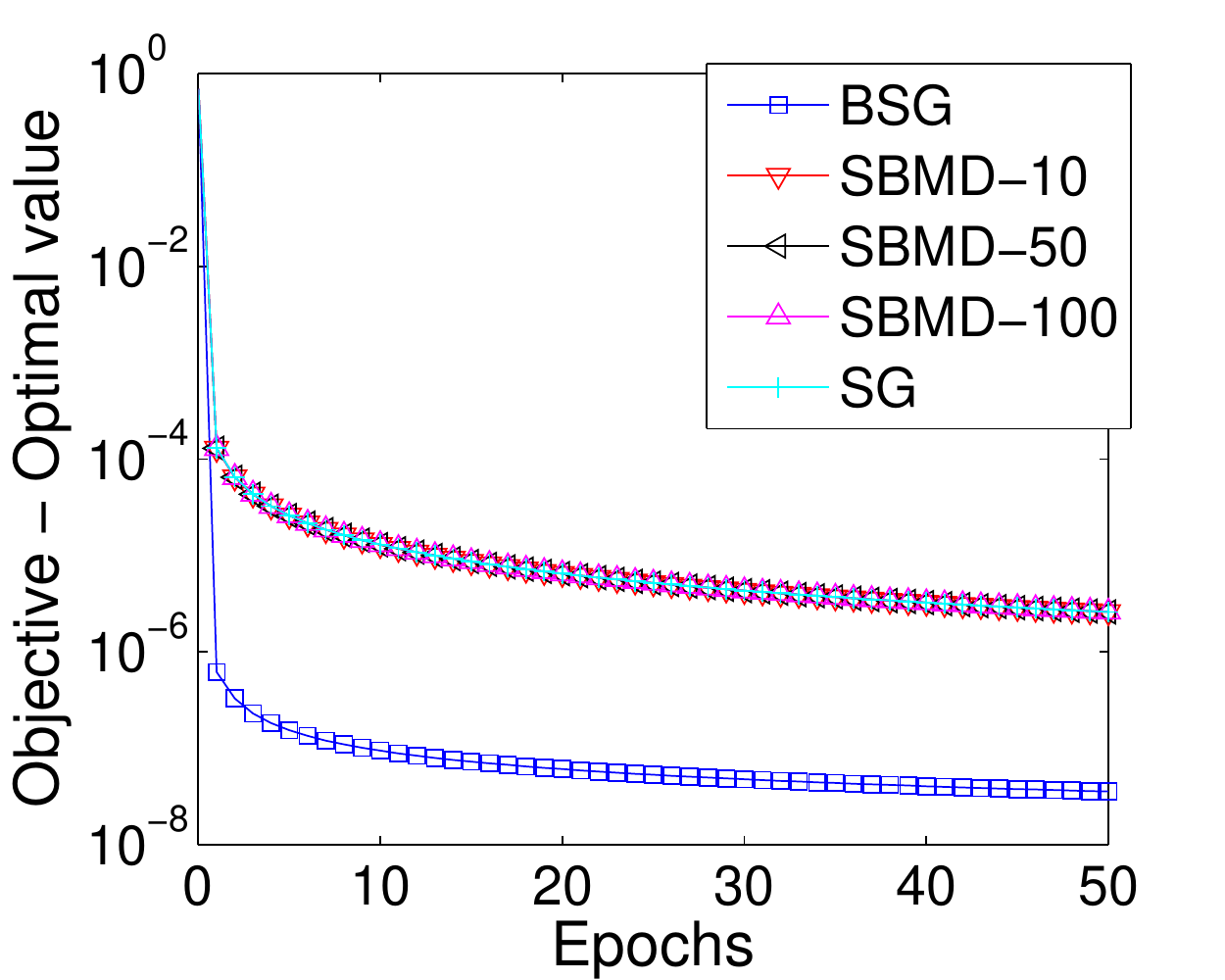}\\
\includegraphics[width = 0.31\textwidth]{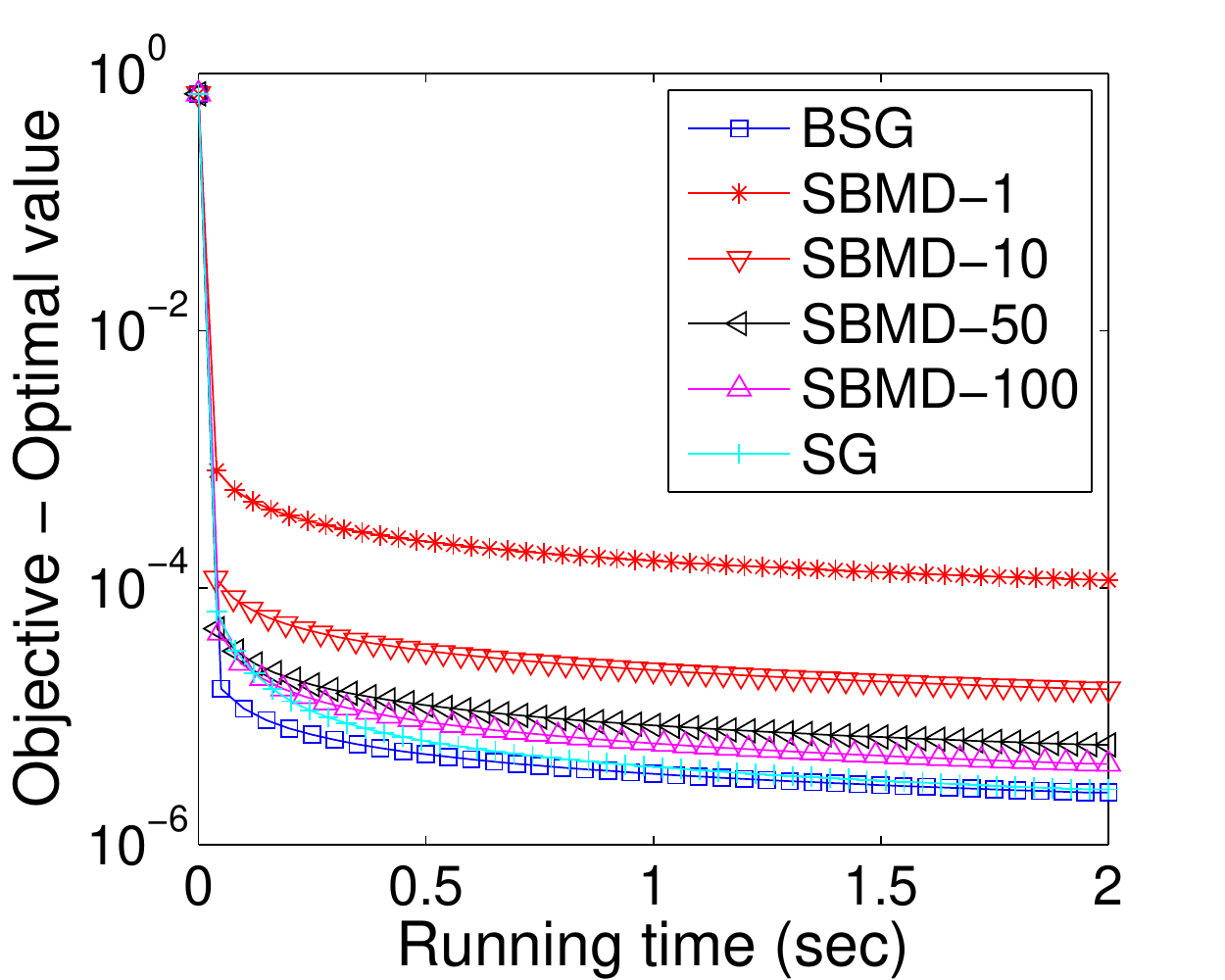}&
\includegraphics[width = 0.31\textwidth]{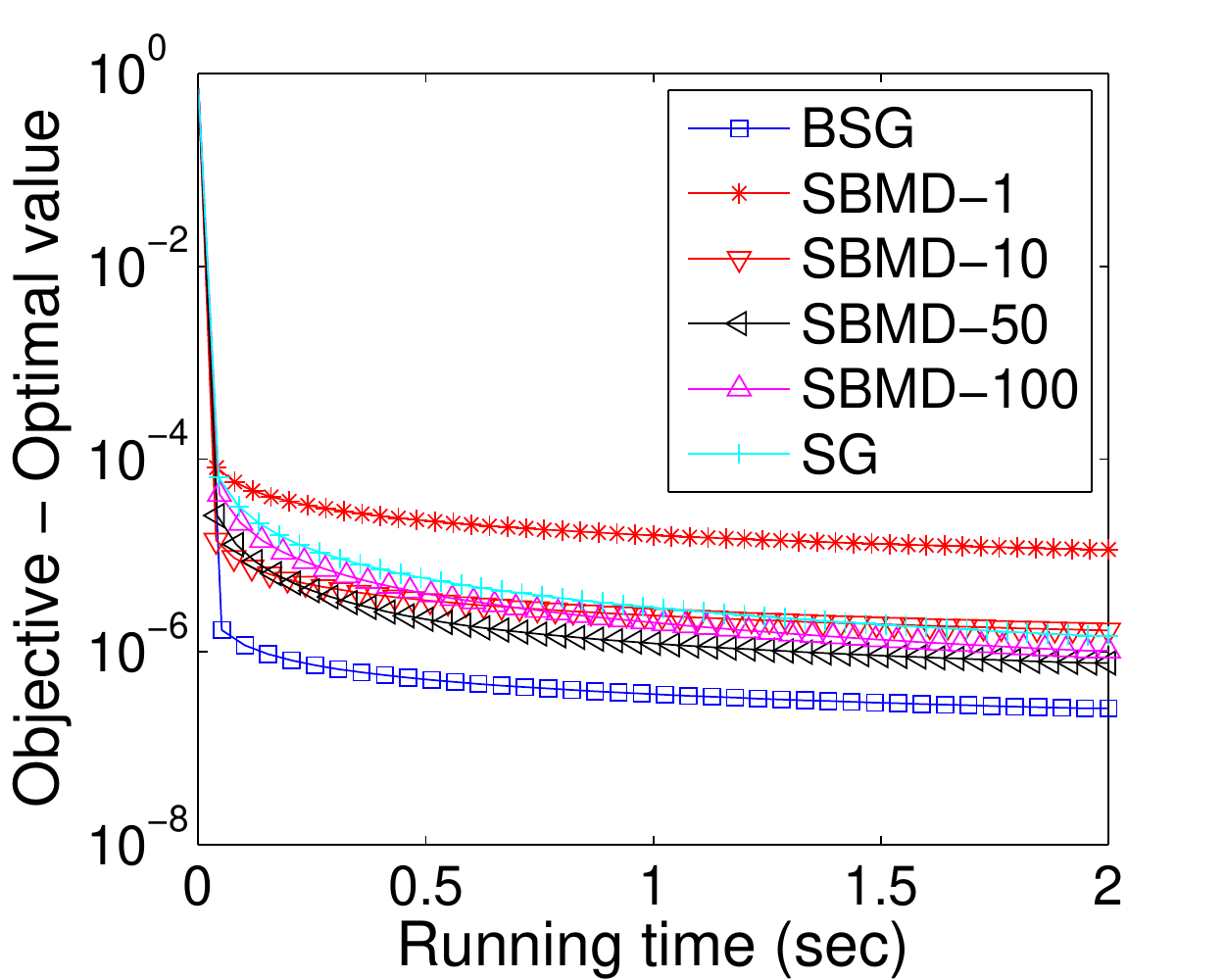}&
\includegraphics[width = 0.31\textwidth]{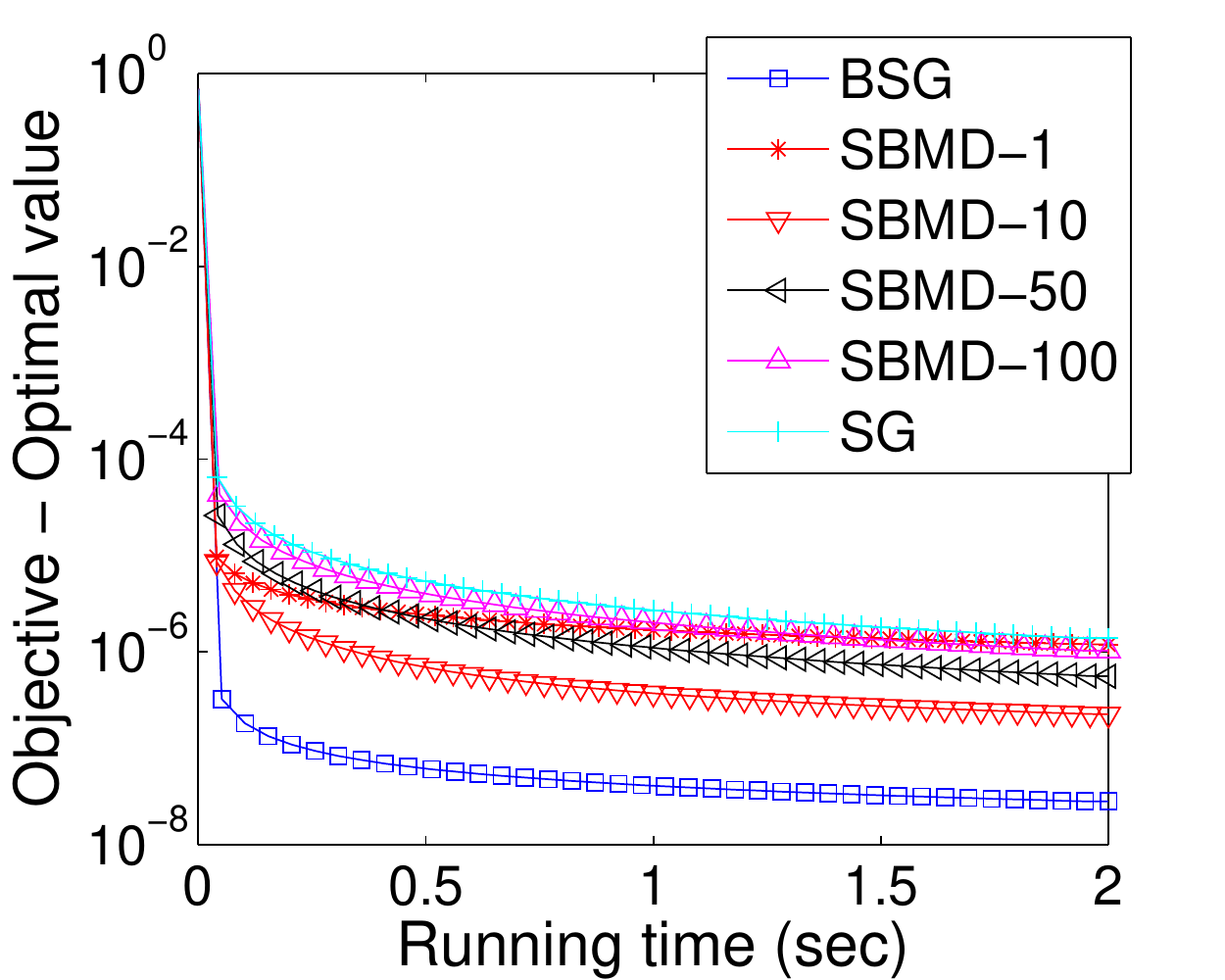}
\end{tabular}
\end{figure}

Secondly, we compared the three algorithms on the \textbf{gisette} dataset\footnote{Available from http://www.csie.ntu.edu.tw/$\sim$cjlin/libsvmtools/datasets/}, which has 6,000 training samples of dimension 5,000. At each iteration, SBMD updated 1,  1,000 or 3,000 coordinates. Figure \ref{fig:lreg-gisette} plots the results of the three algorithms on problem \eqref{eq:lreg} with three different values of $\theta$. From the figures, we see that when $\theta=0.1$, BSG performed almost the same as SG with respect to the number of epochs, and they both outperformed SBMD within the same number of epochs. When $\theta$ becomes larger, BSG reaches much lower objective values than SG and SBMD, and the latter two performed almost the same except SBMD-1.  

\begin{figure}\caption{Objective values of BSG, SG, and SBMD for solving logistic regression \eqref{eq:lreg} on the \textbf{gisette} dataset.}\label{fig:lreg-gisette}
\centering
\begin{tabular}{ccc}
$\theta=0.1$ & $\theta = 1$  & $\theta = 10$\\[-0.2cm]
\includegraphics[width = 0.31\textwidth]{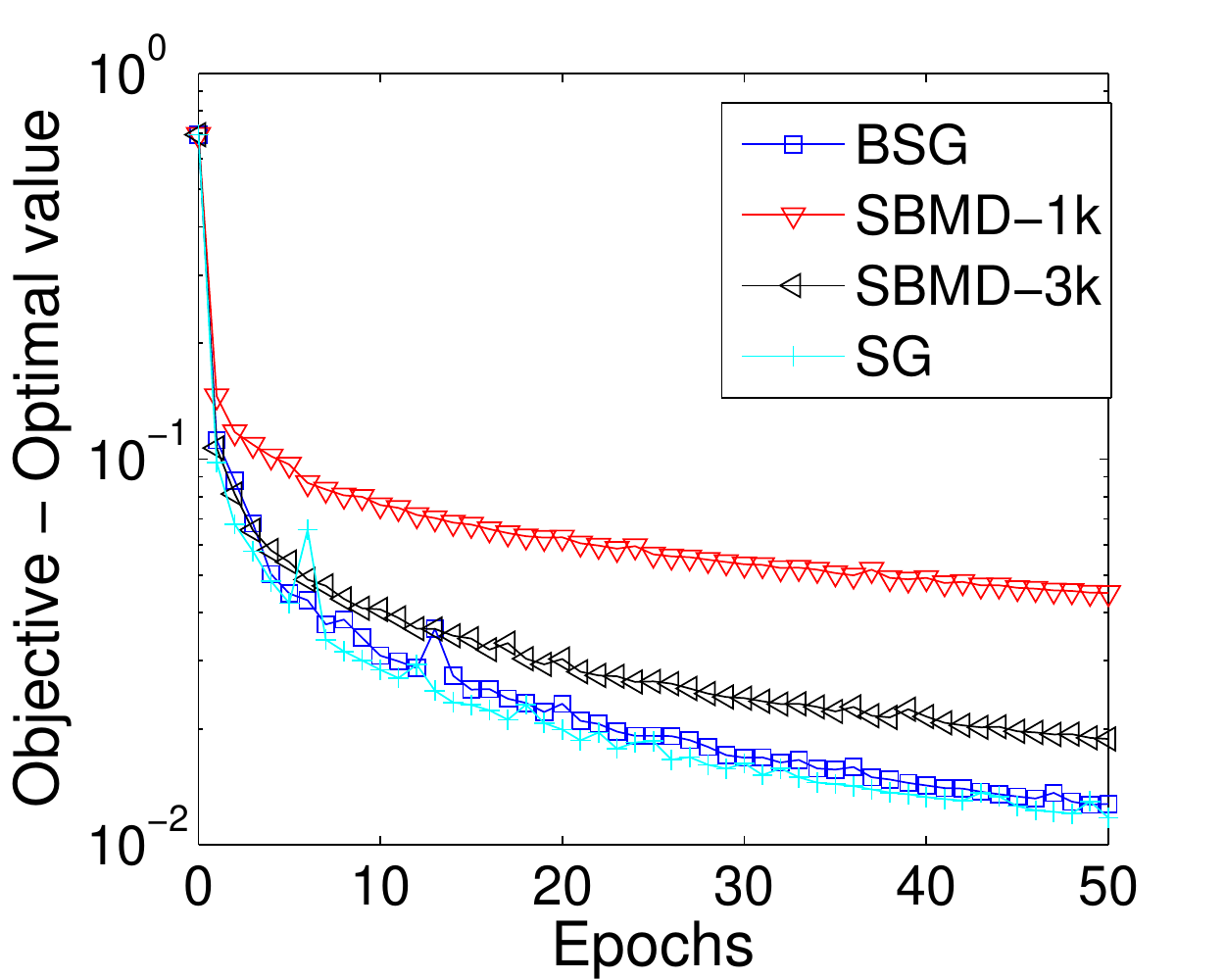}&
\includegraphics[width = 0.31\textwidth]{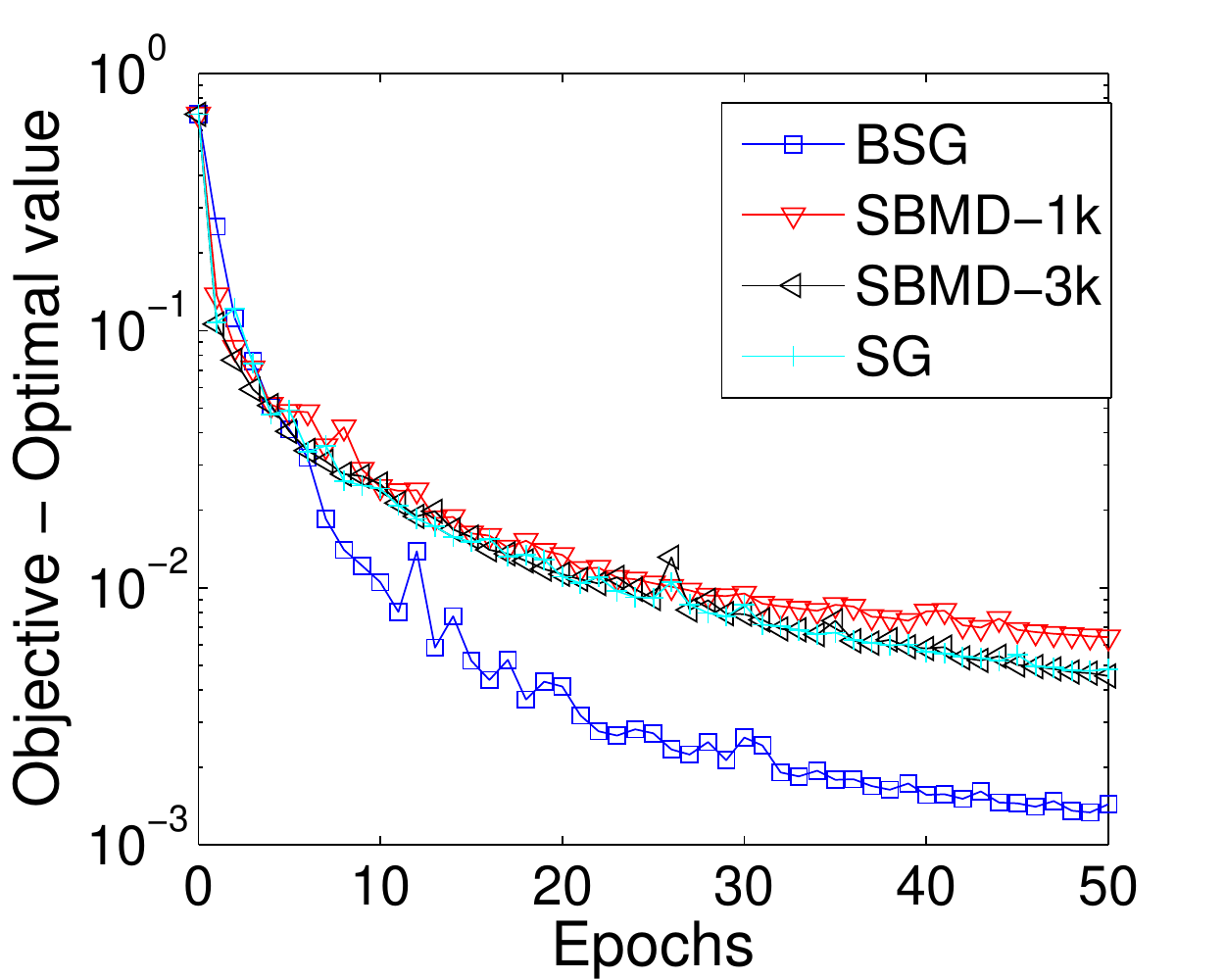}&
\includegraphics[width = 0.31\textwidth]{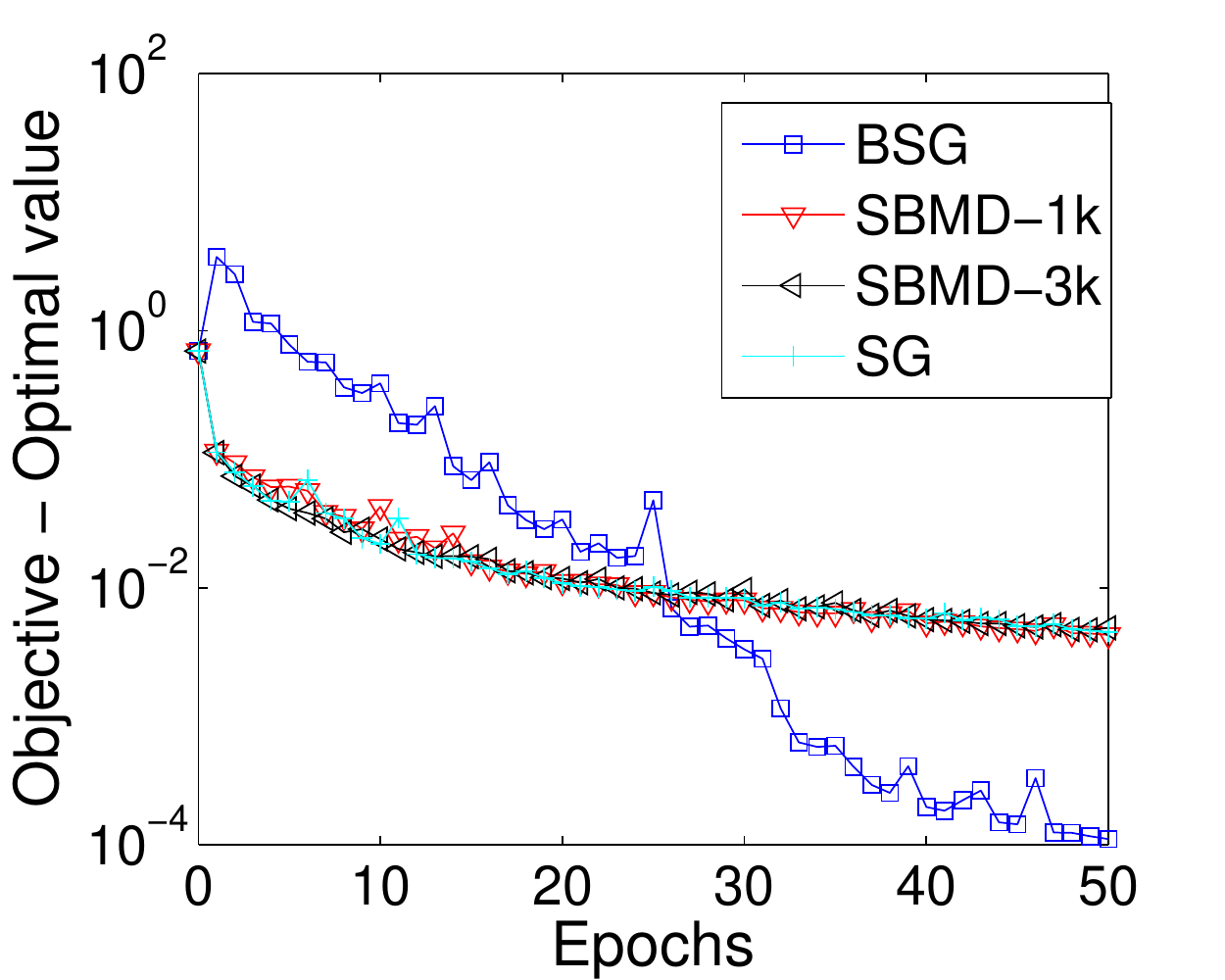}\\
\includegraphics[width = 0.31\textwidth]{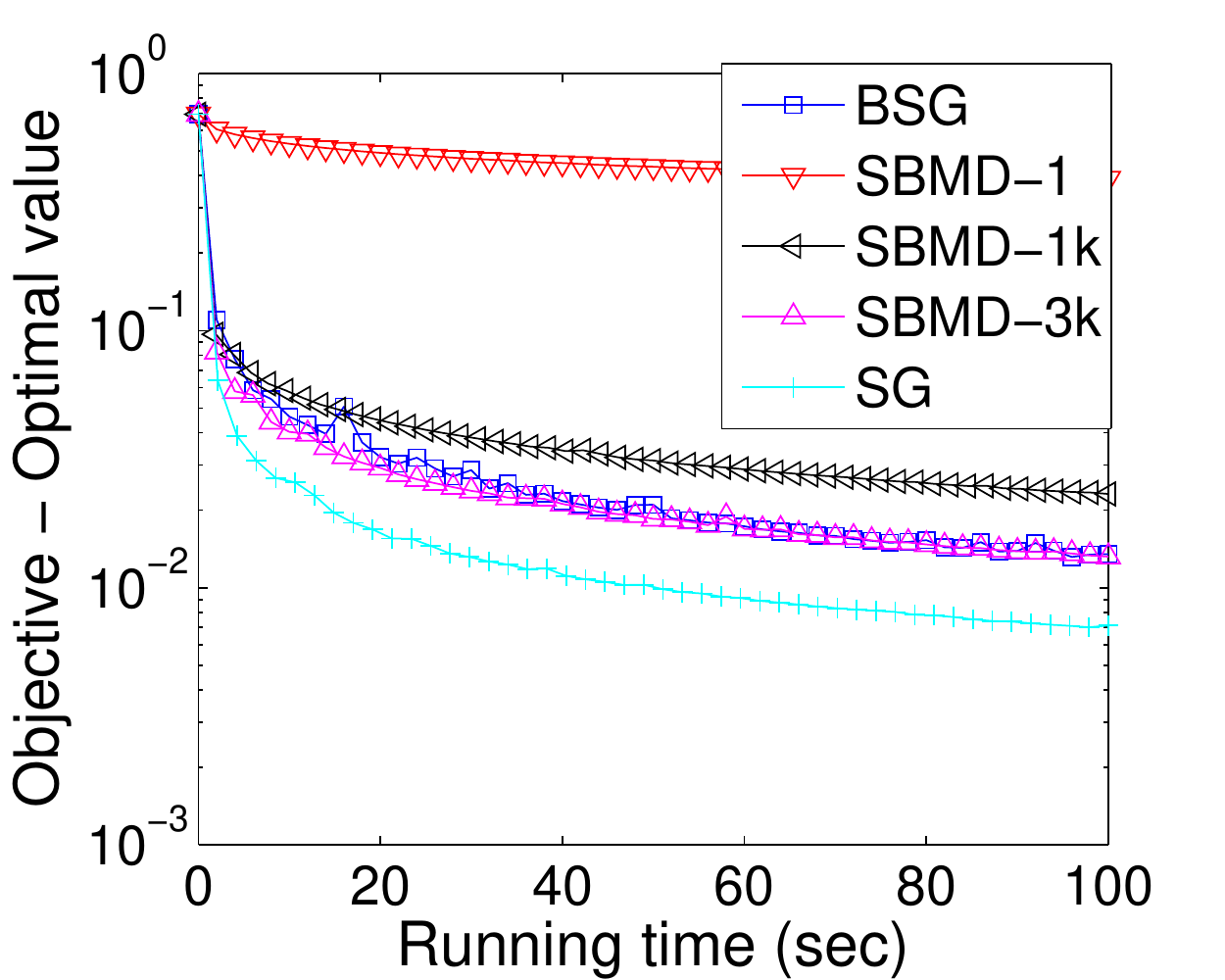}&
\includegraphics[width = 0.31\textwidth]{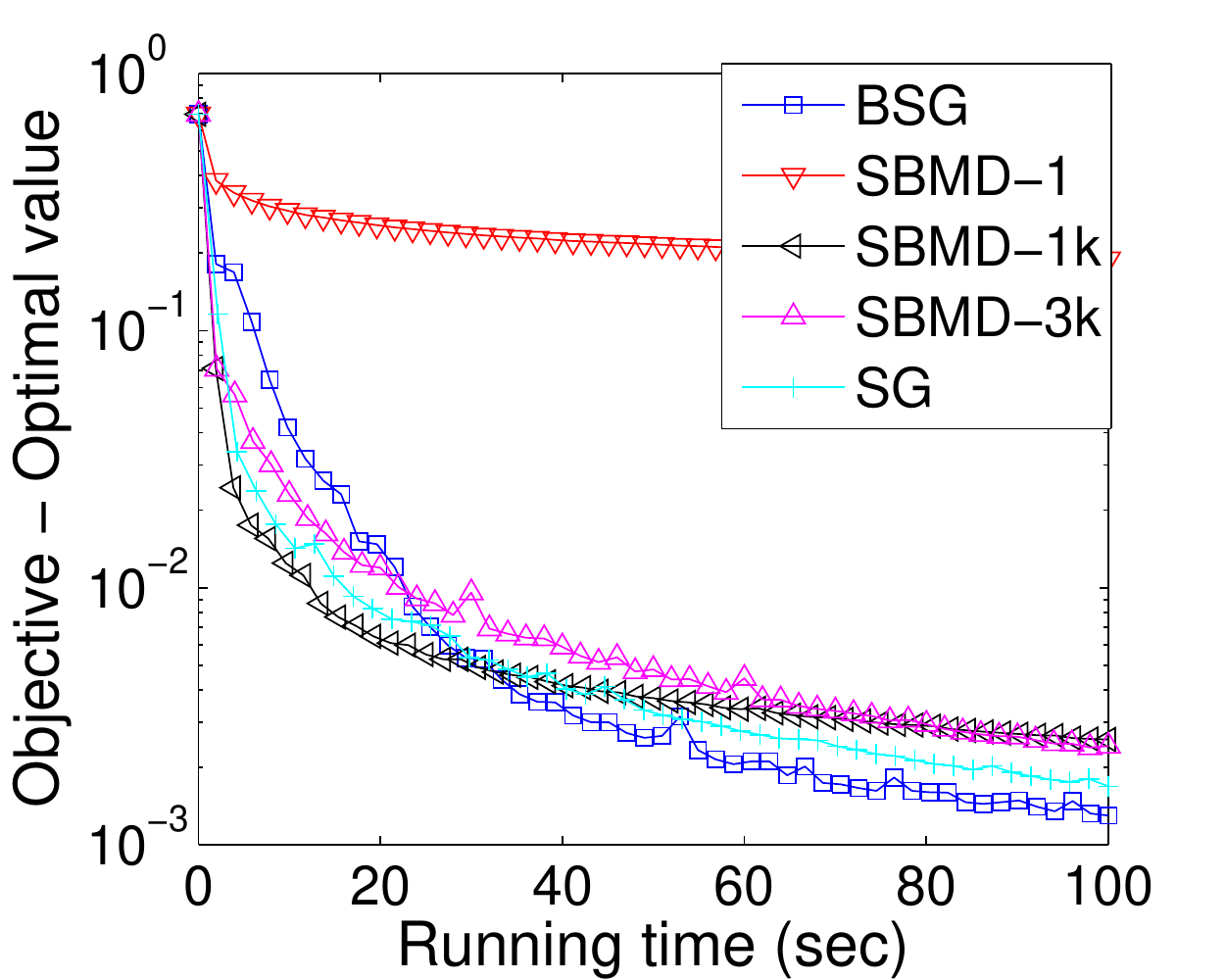}&
\includegraphics[width = 0.31\textwidth]{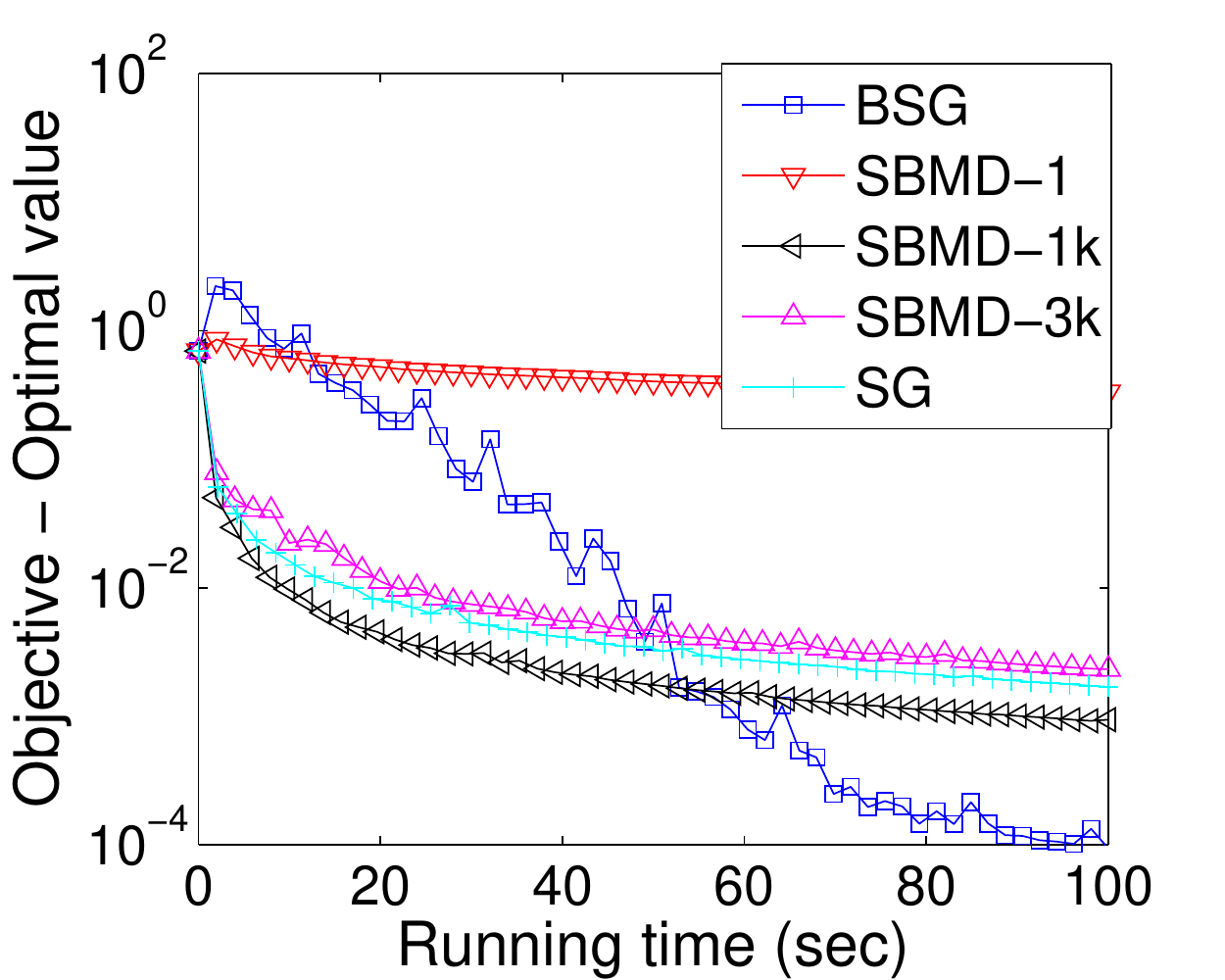}
\end{tabular}
\end{figure}


\subsection{Low-rank tensor recovery}
We compared BSG and BCGD on  the  problem:
\begin{equation}\label{eq:lrtc}
\min_\vX \frac{1}{2N}\sum_{\ell=1}^N(\cA_\ell(\vX_1\circ\vX_2\circ\vX_3)-b_\ell)^2+\sum_{i=1}^3\lambda\|\vX_i\|_1,
\end{equation}
where $b_\ell=\cA_\ell(\bm{\cM})=\langle \bm{\cG}_\ell,\bm{\cM}\rangle, \ell=1,\ldots,N$ with $\bm{\cM}$ being the underlying low-rank tensor. We generated each element of $\bm{\cG}_\ell$ according to the standard Gaussian distribution.
In Figure \ref{fig:3D-shape}, we tested BSG on recovering a low-rank tensor of size $60\times60\times60$ from $N=40,000$ Gaussian random measurements by solving \eqref{eq:lrtc} with $\lambda=0$. The original tensor had 10 middle slices of all \emph{one}'s along each mode and all the other elements to be \emph{zero}. For this test, the data size exceeded the memory of our workstation, and thus BCGD could not be tested.  Figure \ref{fig:3D-shape} plots the original tensor and the recovered one by BSG with 50 epochs and sample size $m_k=256,\forall k$. Its relative error is about 1.93\%.

In Figure \ref{fig:3crs}, we compared BSG and BCGD for solving
\eqref{eq:lrtc} on a smaller tensor of size $32\times 32\times32$. It has the same shape as those in Figure \ref{fig:3D-shape}, and it has 6 middle slices of all \emph{one}'s along each mode and all other elements to be \emph{zero}. We generated $N=15,000$ Gaussian random measurements. Since BSG and BCGD are both based on block coordinate update, they have almost the same\footnote{BSG has slightly higher complexity because of random sampling.} per-epoch complexity, and thus we only plot their objective values with respect to the number of epochs in Figure \ref{fig:3crs}. The left plot shows the objectives by BSG with $m_k=64,\forall k$ and BCGD for solving \eqref{eq:lrtc} with $\lambda=0$, and the right plot corresponds to $\lambda=\frac{1}{N}$ and $m_k=64+\lceil\frac{k-1}{10}\rceil,\forall k$ for BSG. From the figure, we see that BSG significantly outperformed BCGD in the beginning. In the smooth case, BCGD was trapped at some local solution, and in the nonsmooth case, BCGD eventually reached slightly lower objective than that of BSG.

\begin{figure}\caption{3D shape of the original $60\times60\times60$ tensor $\bm{\cM}$ (on the left) and the corresponding recovered one (on the right) by BSG.}\label{fig:3D-shape}
\centering
\begin{tabular}{cc}
\includegraphics[width = 0.35\textwidth]{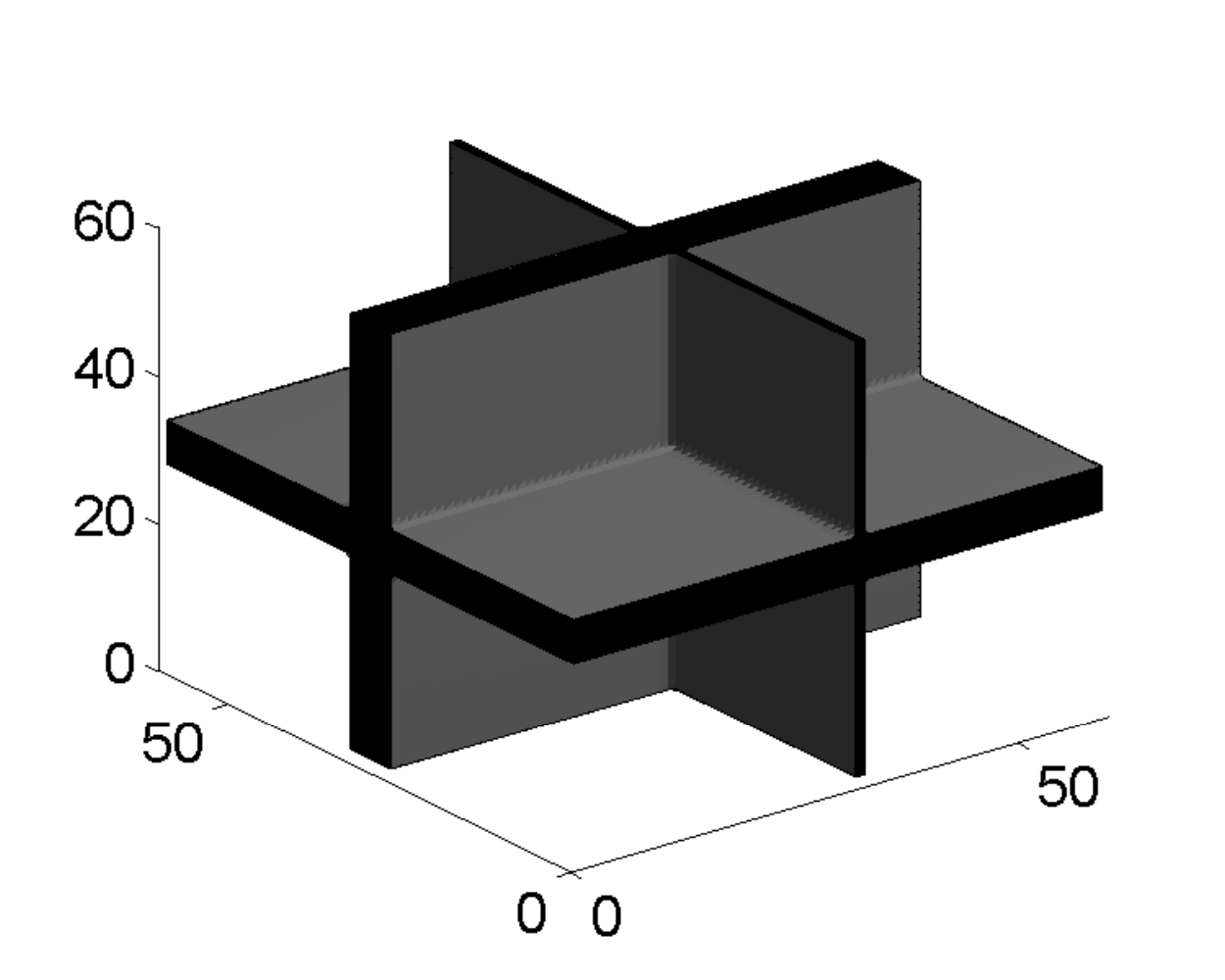} &
\includegraphics[width = 0.35\textwidth]{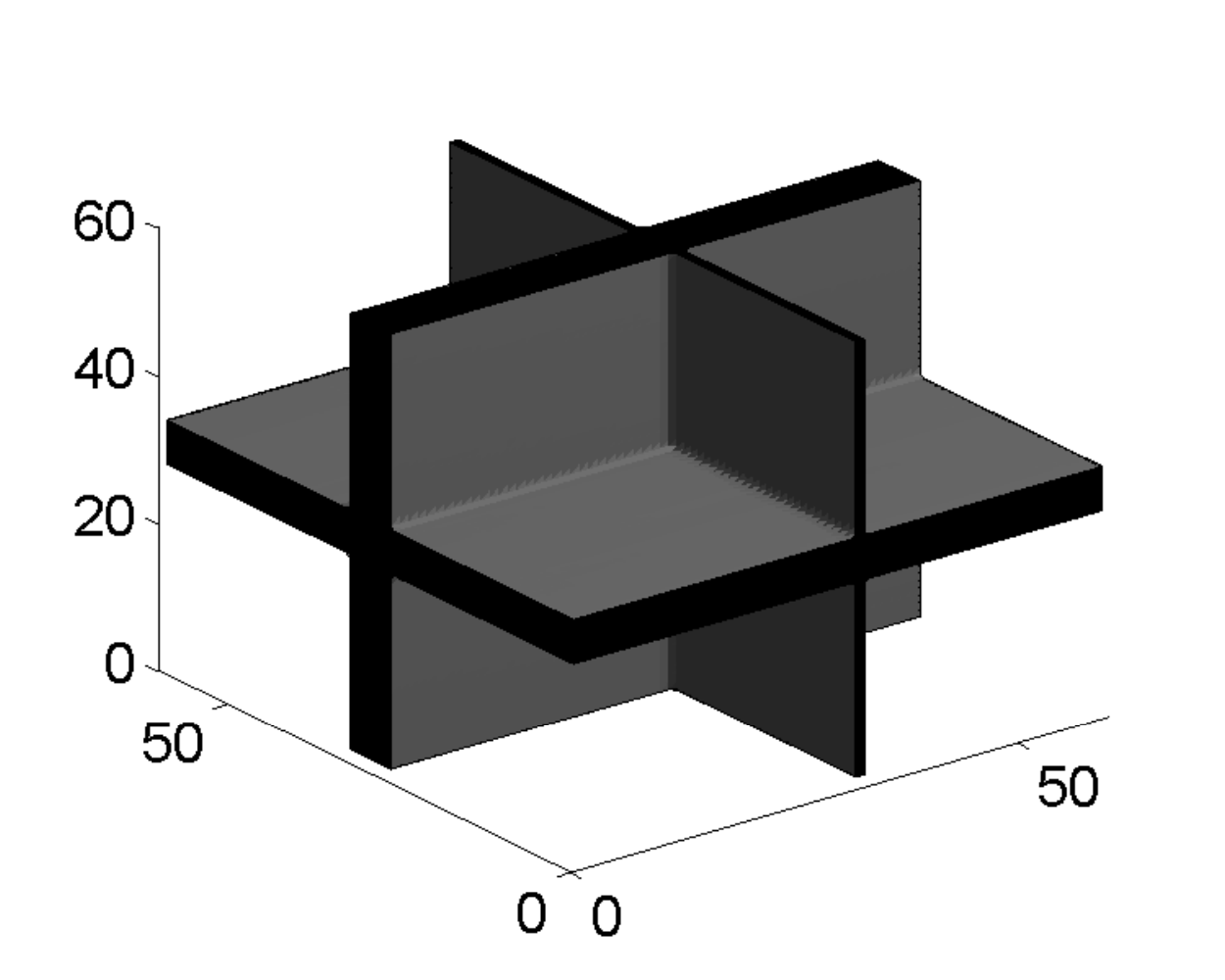}
\end{tabular}
\end{figure}

\begin{figure}\caption{Objective values of BSG and BCGD for solving \eqref{eq:lrtc} on a $32\times32\times32$ low-rank tensor. Left: $\ell_1$ penalty parameter $\lambda=0$; Right: $\lambda=\frac{1}{N}$}\label{fig:3crs}
\centering
\includegraphics[width = 0.35\textwidth]{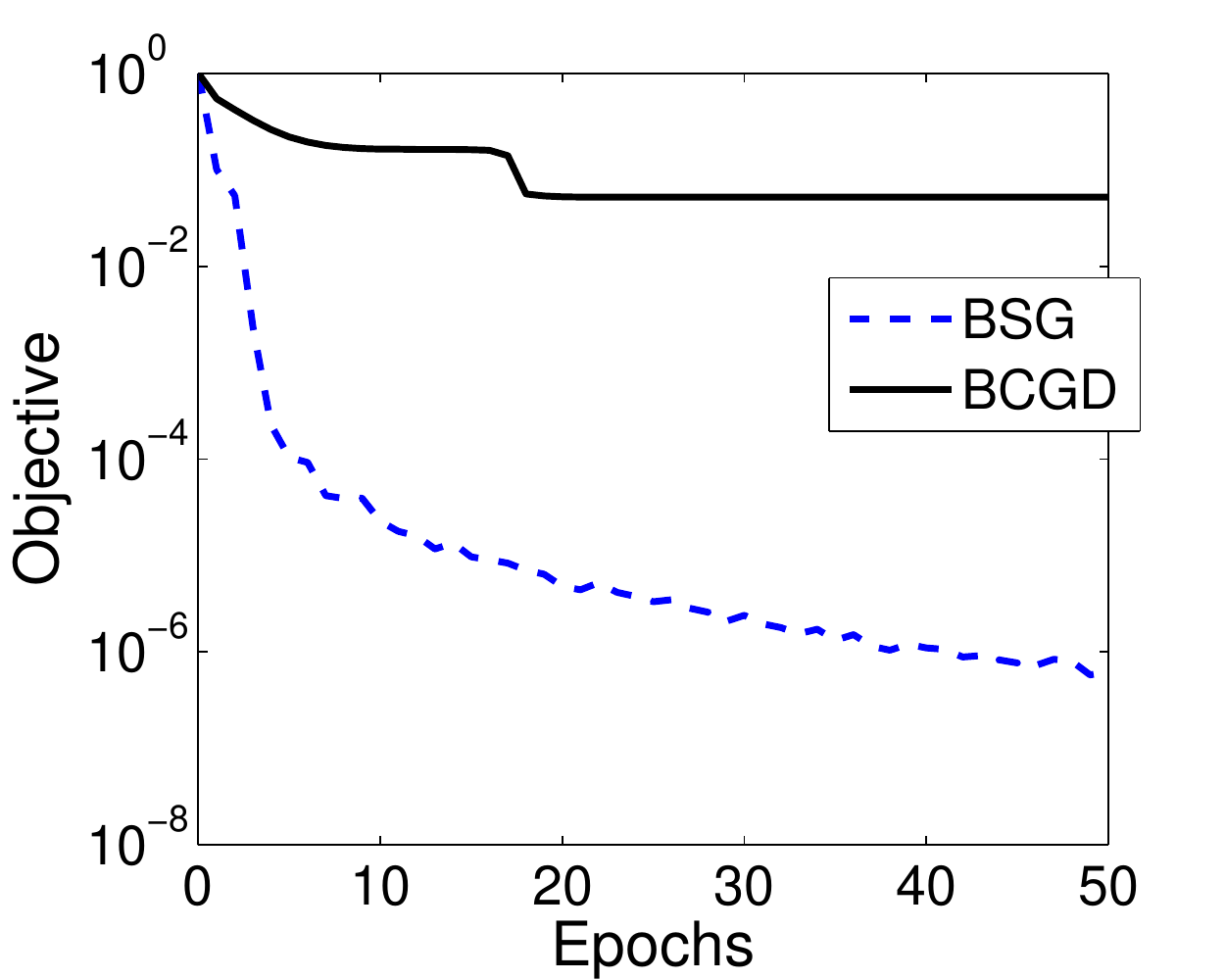}
\includegraphics[width = 0.35\textwidth]{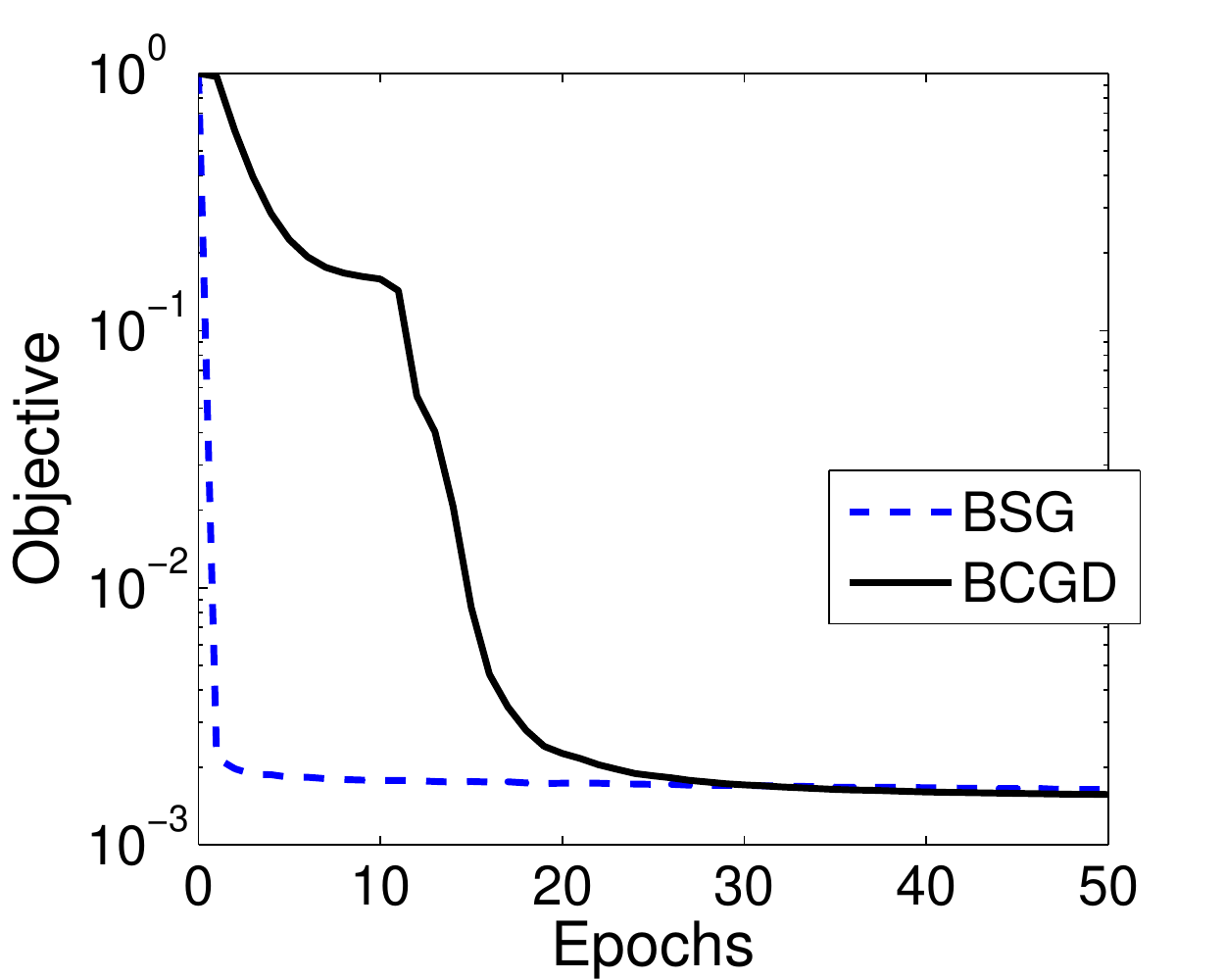}
\end{figure}

\subsection{Bilinear logistic regression}
We compared BSG and BCGD on  the  problem:
\begin{equation}\label{eq:blr}
\min_{\vU,\vV,b} \frac{1}{N}\sum_{\ell=1}^N\log\big(1+\exp[-y_\ell(\tr(\vU^\top\vX_\ell\vV)+b)]\big),
\end{equation}
where $\{(\vX_\ell,y_\ell)\},\ell=1,\ldots,N,$ were given training samples with class labels $y_\ell\in\{+1,-1\}$. The bilinear logistic regression appears to be first used in \cite{dyrholm2007bilinear} for EEG data classification. It applies the matrix format of the original data, in contrast to the standard linear logistic regression which collapses each feature matrix into a vector. It has been shown that the bilinear logistic regression outperforms the standard linear logistic regression in many applications such as brain-computer interface \cite{dyrholm2007bilinear} and visual recognition \cite{shi2014sparse}.

In this test, we used the EEG dataset IVb from BCI competition III\footnote{http://www.bbci.de/competition/iii/} and the dataset concerns motor imagery with uncued classification task. The 118 channel EEG was recorded from a healthy subject sitting in a comfortable chair with arms resting on armrests. Visual cues (letter presentation) were shown for 3.5 seconds, during which the subject performed: left hand, right foot, or tongue. The data was sampled at 100 Hz, and the cues of ``left hand'' and ``right foot'' were marked in the training data. We chose all the 210 marked data points, and for each data point We randomly subsampled 100 temporal slices independently for 10 times to get, in total, 2,100 samples of size $118\times 100$.

We set $m_k=64,\forall k$ in \eqref{eq:appf} for BSG and used the same random starting point for BSG and BCGD. The left plot of Figure \ref{fig:bci} depicts their convergence behaviors. From the figure, we see that BSG significantly outperformed BCGD within 50 epochs. Running BCGD to more epochs, we observed that BCGD could later reach a similar objective as that of BSG. The right plot of Figure \ref{fig:bci} shows the prediction accuracy of the solutions of BSG and BCGD, both of which ran to 30 epochs, and  that of the result returned by LIBLINEAR \cite{fan2008liblinear}, which solved linear logistic regression to its default tolerance. We ran the three methods 20 times. For each run, we randomly chose 2,000 samples for training and the remaining ones for testing. From the figure, we see that the \emph{bilinear} logistic regression problem solved by BSG gave consistently higher prediction accuracies than the \emph{linear} logistic regression problem. The low accuracies given by BCGD were results of its non-convergence in 30 epochs, which can be observed from the left plot of Figure \ref{fig:bci}. Running to more epochs, BCGD will eventually give similar predictions as  BSG. However, that will take much more time.

\begin{figure}\caption{Left: objective values of  \eqref{eq:blr} (lower is better) given by BSG and BCGD on the BCI EEG data within 50 epochs; Right: prediction accuracies (higher is better) of 20 independent runs by BSG and BCGD within 30 epochs and also by LIBLINEAR.}\label{fig:bci}
\centering
\includegraphics[height = 0.3\textwidth]{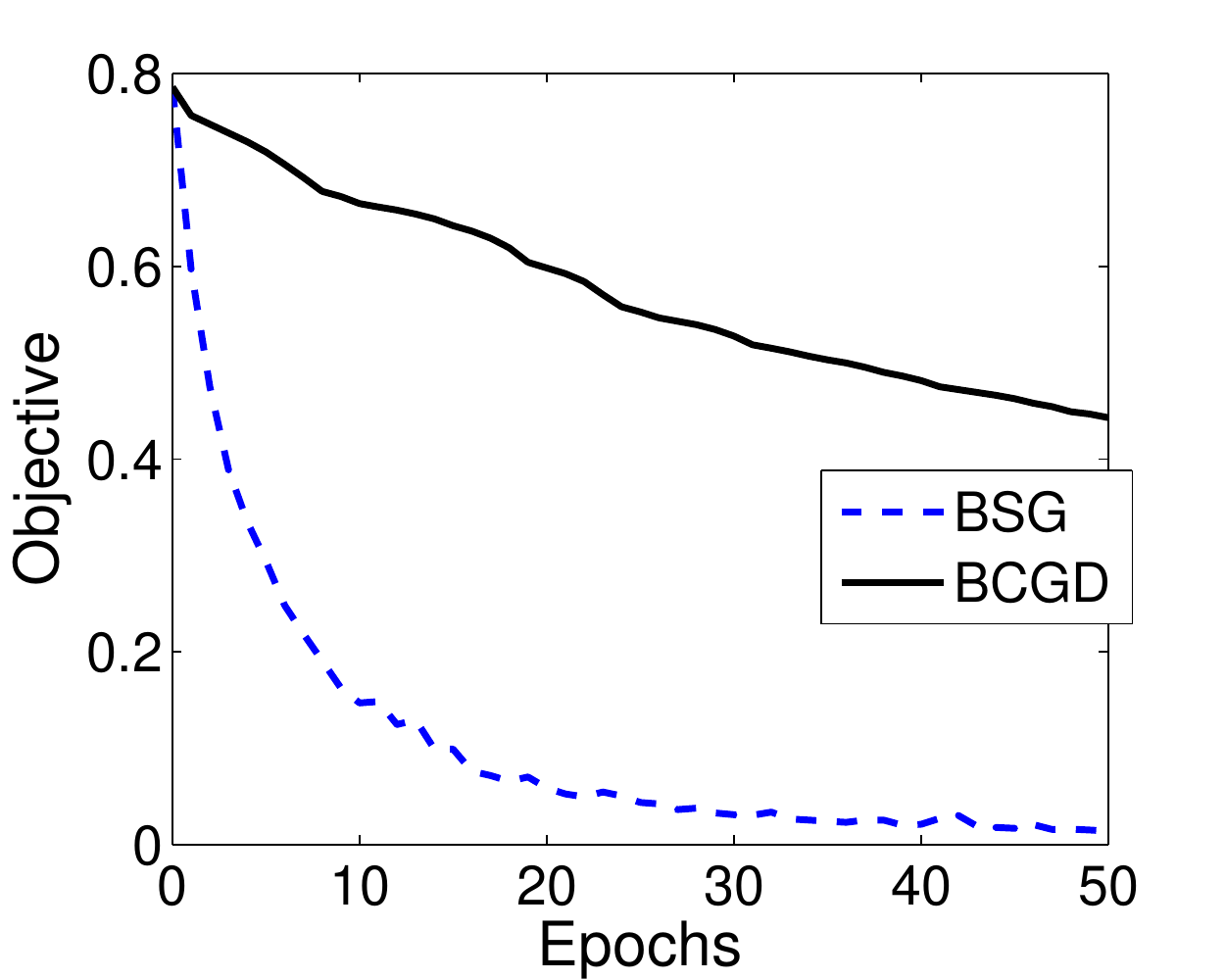}
\includegraphics[height = 0.3\textwidth]{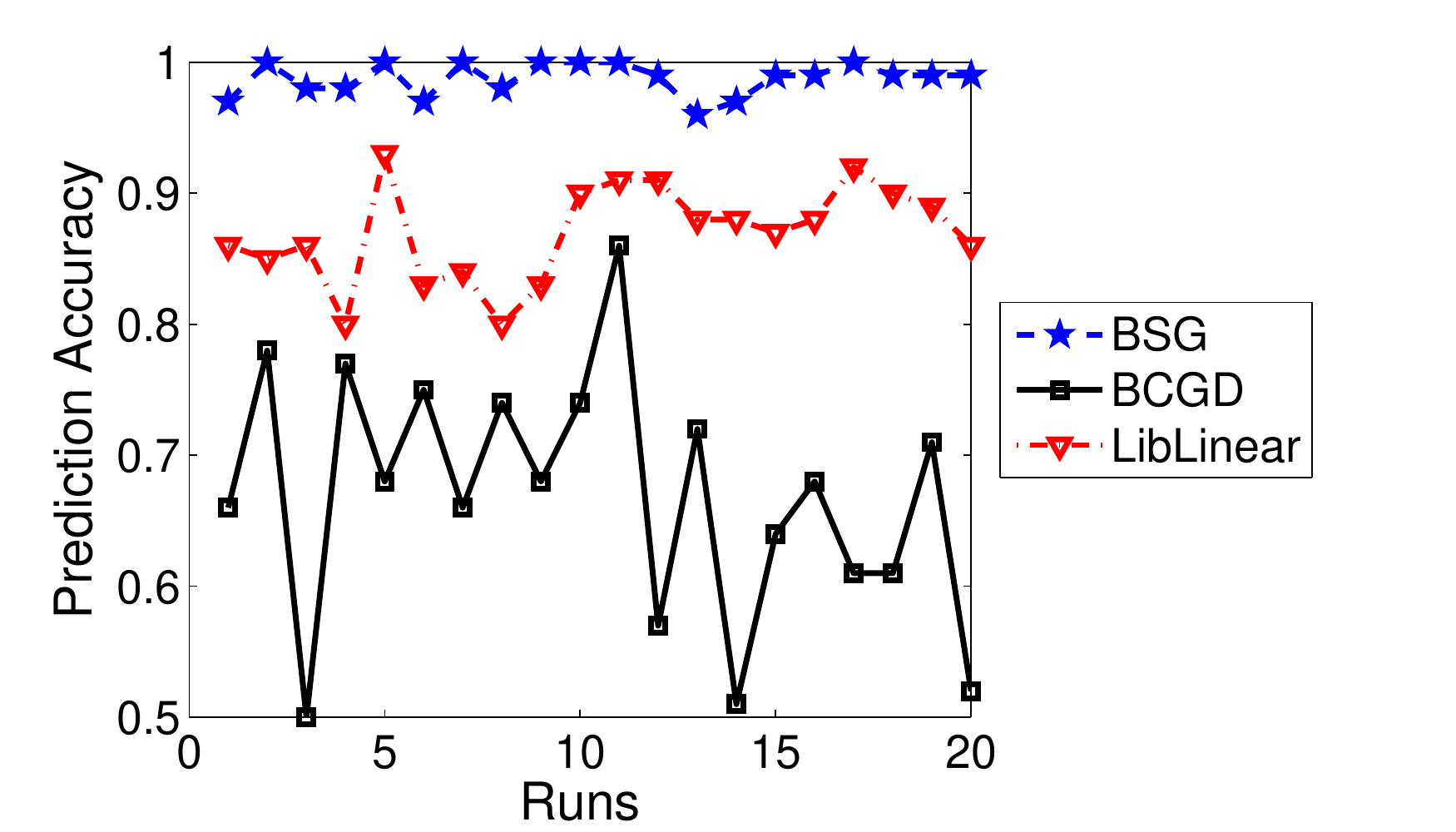}
\end{figure}

\section{Conclusions}
We have proposed a BSG (block stochastic gradient) method and analyzed its convergence for both convex and nonconvex problems. The method has a convergence rate similar to that of the SG (stochastic gradient) method for convex programming, and its convergence has been established in terms of the expected violation of first-order optimality conditions for the nonconvex case. Numerical results demonstrate its clear advantages over SG and a BSMD (block stochastic mirror descent) method on the tested convex problems and  over the BCGD (block coordinate gradient descent) method one the tested nonconvex problems.

\section*{Acknowledgements}
This work was supported in part by NSF grant DMS-1317602 and ARO MURI grant W911NF-09-1-0383.

\appendix
\section{Proof of some lemmas}\label{sec:app}
We give the proofs of some  lemmas in the paper.
\subsection{Proof of Lemma \ref{lem:indp}}
The result in \eqref{eq:indp} can be shown by
\begin{align*}
\EE\langle \vu^k, \bm{\delta}_i^k\rangle = & \EE_{\bm{\Xi}_{[k-1]}}\left[\EE\big[\langle \vu^k, \bm{\delta}_i^k\rangle|\bm{\Xi}_{[k-1]}\big]\right]\\
=&\EE_{\bm{\Xi}_{[k-1]}}\left[\big\langle \EE\big[\vu^k|\bm{\Xi}_{[k-1]}\big], \EE\big[\bm{\delta}_i^k|\bm{\Xi}_{[k-1]}\big]\big\rangle\right]\\
\le &\EE_{\bm{\Xi}_{[k-1]}}\left[\big\|\EE[\vu^k|\bm{\Xi}_{[k-1]}]\big\|\cdot\big\|\EE[\bm{\delta}_i^k|\bm{\Xi}_{[k-1]}]\big\|\right]\\
\le & A(\max_j\alpha_j^k)\EE_{\bm{\Xi}_{[k-1]}}\left[\big\|\EE[\vu^k|\bm{\Xi}_{[k-1]}]\big\|\right]\\
\le & A(\max_j \alpha_j^k)\EE\|\vu^k\|,
\end{align*}
where the second equality follows from the conditional independence between $\vu^k$ and $\bm{\delta}_i^k$, and the last inequality follows from the Jensen's inequality.

\subsection{Proof of Lemma \ref{lem:bdh}}
We first show the following lemma.
\begin{lemma}\label{lem:bdsubdif}
For any function $\psi$ and positive scalar $\alpha$, it holds that 
\begin{equation}\label{eq:boundh}
\|\frac{1}{\alpha}\big(\vx-\prox_{\alpha \psi}(\vx-\alpha\vy)\big)\|^2\le 2\|\vy\|^2+2\|\partial \psi(\prox_{\alpha \psi}(\vx-\alpha\vy))\|^2,
\end{equation}
where $\|\partial\psi(\vz)\|=\sup_\vu\{\|\vu\|: \vu\in\partial \psi(\vz)\}$, and if $\partial \psi(\vz)=\emptyset$, we let $\|\partial\psi(\vz)\|=+\infty$ by convention. 
\end{lemma}
\begin{proof}
Let $\vz=\prox_{\alpha \psi}(\vx-\alpha\vy)$. Then 
$\vzero\in\vy+\frac{1}{\alpha}(\vz-\vx)+\partial\psi(\vz),$ namely, for some $\vu\in\partial\psi(\vz)$, it holds $\frac{1}{\alpha}(\vx-\vz)=\vy+\vu$. Hence,
$$\|\frac{1}{\alpha}(\vx-\vz)\|^2\le 2\|\vy\|^2+2\|\vu\|^2,$$
which completes the proof.
\end{proof}

Now we are ready to prove Lemma \ref{lem:bdh}.
For $i\in\cI_1$, we have
$$\tilde{\vh}_i^k=\frac{1}{\alpha_i^k}\left(\vx_i^k-\prox_{\alpha_i^k r_i}(\vx_i^k-\alpha_i^k\tilde{\vg}_i^k)\right),$$
and thus from Lemma \ref{lem:bdsubdif} and Remark \ref{rm:exlip}, it follows that
$$\EE\|\tilde{\vh}_i^k\|^2\le 2\EE\|\tilde{\vg}_i^k\|^2+2L_{r_i}^2\le 4\EE\|\vg_i^k\|^2+4\EE\|\bm{\delta}_i^k\|^2+2L_{r_i}^2\le 4M_\rho^2+4\sigma_k^2+2L_{r_i}^2,$$
where we have used the Cauchy-Schwarz inequality in the second inequality.
For $i\in\cI_2$, we have
\begin{align*}
\EE\|\tilde{\vh}_i^k\|^2=&\frac{1}{(\alpha_i^k)^2}\EE\left\|\vx_i^k-\cP_{\cX_i}\big(\vx_i^k-\alpha_i^k(\tilde{\vg}_i^k+\tilde{\nabla} r_i(\vx_i^k))\big)\right\|^2\\
\le &\EE\|\tilde{\vg}_i^k+\tilde{\nabla} r_i(\vx_i^k)\|^2\le 4M_\rho^2+4\sigma_k^2+2L_{r_i}^2,
\end{align*}
where we have used the nonexpansiveness of the projection operator in the first inequality. This completes the proof of Lemma \ref{lem:bdh}.

\subsection{Proof of Lemma \ref{lem:subrate}}
The result can be proved by induction. First, consider the case of $a\le 1$. When $k=1$, \eqref{eq:subrate} obviously holds. Suppose it holds for some $k\ge 1$. Then 
\begin{align*}
A_{k+1}-\frac{c}{k+1}\le &\big(1-\frac{a}{k}\big)\frac{c}{k}+\frac{b}{k^2}-\frac{c}{k+1}\\
= & \frac{1}{k^2}\left(b-\big(a+\frac{k}{k+1}\big)c\right)\\
\le &\frac{1}{k^2}\left(b-\big(a+\frac{k}{k+1}\big)\frac{b}{a}\right) < 0,
\end{align*}
which shows $A_{k+1}\le c/(k+1)$.

Secondly, consider $a>1$. When $k=\lfloor a \rfloor+1$, it holds from
$$A_{\lfloor a \rfloor+1}\le \left(1-\frac{a}{\lfloor a \rfloor}\right)A_{\lfloor a \rfloor}+\frac{b}{(\lfloor a \rfloor)^2}\le \frac{b}{(\lfloor a \rfloor)^2}\le \frac{2b}{(\lfloor a \rfloor+1)(a-1)}.$$
Suppose \eqref{eq:subrate} holds for some $k\ge \lfloor a \rfloor+1$. Then 
\begin{align*}
A_{k+1}-\frac{c}{k+1}\le &\big(1-\frac{a}{k}\big)\frac{c}{k}+\frac{b}{k^2}-\frac{c}{k+1}\\
= & \frac{1}{k^2}\left(b-\big(a+\frac{k}{k+1}\big)c\right)\\
\le &\frac{1}{k^2}\left(b-\big(a+\frac{k}{k+1}\big)\frac{2b}{a-1}\right) < 0,
\end{align*}
which indicates $A_{k+1}\le c/(k+1)$ and completes the proof.

\subsection{Proof of Lemma \ref{lem:limmax}}
Suppose $\lim_{k\to\infty}a_{k+1}/a_k=\nu$. We prove the result for the cases of $\nu<\eta$ and $\nu=\eta$, and the case of $\nu>\eta$ can be shown in a similar way as that of $\nu<\eta$. 

\textbf{Case 1:} $\nu<\eta$. Since $\lim_{k\to\infty}a_{k+1}/a_k = \nu$, for $\epsilon=\frac{\eta-\nu}{2}>0$, there exists a sufficiently large integer $K_1>0$ such that $a_{k+1}/a_k\le \nu+\epsilon=\frac{\nu+\eta}{2},\,\forall k\ge K_1$. Therefore, $a_k\le \big(\frac{\nu+\eta}{2}\big)^{k-K_1}a_{K_1},\,\forall k\ge K_1$. Since $\nu<\eta$, we can choose another sufficiently large integer $K\ge K_1$ to have $a_k\le \eta^k,\,\forall k\ge K.$  Hence, $e_k=\eta^k,\,\forall k\ge K$, and the limit of $e_{k+1}/e_k$ is $\eta$. 

\textbf{Case 2:} $\nu=\eta$. If $a_k\ge \eta^k,\forall k$, or $a_k\le \eta^k, \forall k$, then the result is obvious. Otherwise, there must exist a sequence $\{n_\ell\}_{\ell=1}^\infty$ such that
$$\left\{\begin{array}{ll}
a_k\ge \eta^k, & \text{ if } n_{2m-1}\le k< n_{2m}\\[0.1cm]
a_k< \eta^k, & \text{ if } n_{2m}\le k < n_{2m+1}
\end{array}\right.$$
Note that if $n_{2m-1}\le k<k+1< n_{2m}$ or $n_{2m}\le k<k+1< n_{2m+1}$, then it is easy to have $e_{k+1}/e_k\to\eta$ as $k\to\infty$. In addition,
$$\frac{e_{n_{2m}}}{e_{n_{2m-1}}}=\frac{\eta^{n_{2m}}}{a_{n_{2m}-1}}\le \frac{\eta^{n_{2m}}}{\eta^{n_{2m}-1}}=\eta,$$
and 
$$\frac{e_{n_{2m}}}{e_{n_{2m}-1}}=\frac{\eta^{n_{2m}}}{a_{n_{2m}-1}}\ge \frac{a_{n_{2m}}}{a_{n_{2m}-1}}\to\eta,\text{ as }m\to\infty.$$ 
Hence, $\lim_{m\to\infty}e_{n_{2m}}/e_{n_{2m}-1}=\eta$, and in the same way, one can show $$\lim_{m\to\infty}e_{n_{2m+1}}/e_{n_{2m+1}-1}=\eta.$$ Therefore, the limit of $e_{k+1}/e_k$ is $\eta$. This completes the proof.

\bibliographystyle{siam}

\end{document}